\documentclass[arxiv,reqno,twoside,a4paper,12pt]{amsart}

\usepackage[utf8]{inputenc}
\usepackage{amsmath}
\usepackage{amsfonts}
\usepackage{amssymb}
\usepackage{amsthm}
\usepackage{float}
\usepackage{mathtools}
\usepackage{csquotes}
\usepackage{upgreek,esint,enumerate}
\usepackage[numbers]{natbib}
\usepackage{dsfont}
\usepackage{enumitem}
\usepackage{graphicx}

\usepackage[colorlinks=true,linkcolor=blue,citecolor=blue, urlcolor=blue]{hyperref}

\usepackage[scaled]{helvet} 
\usepackage{courier} 
\usepackage[mathbf]{euler}
\usepackage[dvipsnames]{xcolor}

\usepackage[driver=pdftex,margin=3cm,heightrounded=true,centering]{geometry}

\usepackage{tikz-cd}
\usetikzlibrary{arrows,matrix,shapes,decorations.text, mindmap,shapes.misc,decorations.markings}

\usepackage{metalogo}

\newcommand{\N}{\mathbb{N}}
\newcommand{\C}{\mathbb{C}}

\renewcommand{\S}{\mathbb{S}}
\renewcommand{\P}{\mathbb{P}}
\newcommand{\R}{\mathbb{R}}
\newcommand{\norm}[1]{\left\lVert#1\right\rVert}

\newcommand{\tr}[0]{\mathrm{tr}}
\newcommand{\Tr}[0]{\mathrm{Tr}}
\newcommand{\FP}{\textup{F.P.}}
\newcommand{\Vol}[0]{\mathrm{Vol}}
\newcommand{\End}{\mathrm{End}}

\usepackage[numbers]{natbib}
\newcommand{\spec}[0]{\mathrm{spec}}
\newcommand{\Br}[0]{\right) }
\newcommand{\Bl}[0]{\left( }
\renewcommand{\Re}[0]{\mathrm{Re}}
\newcommand{\bcup}[0]{\overline{\cup}}

\DeclareMathOperator*{\LIM}{LIM}

\newcommand{\cU}{\mathscr{U}}

\newcommand{\diag}{\mathrm{diag}}
\newcommand{\dVol}{\mathrm{dVol}}

\tikzset{->-/.style={decoration={markings, mark=at position #1 with {\arrow{>}}},postaction={decorate}}}
\tikzset{cross/.style={cross out, draw=black, minimum size=2*(#1-\pgflinewidth), inner sep=0pt, outer sep=0pt},
cross/.default={1pt}}

\theoremstyle{plain}
\newtheorem{thm}{Theorem}[section]

\newtheorem{Lem}[thm]{Lemma}
\newtheorem{Cor}[thm]{Corollary}
\newtheorem{Prop}[thm]{Proposition}
\newtheorem{Def}[thm]{Definition}
\newtheorem{Rem}[thm]{Remark}

\newtheorem{Assump}[thm]{Assumption}

\def\Xint#1{\mathchoice
{\XXint\displaystyle\textstyle{#1}}%
{\XXint\textstyle\scriptstyle{#1}}%
{\XXint\scriptstyle\scriptscriptstyle{#1}}%
{\XXint\scriptscriptstyle\scriptscriptstyle{#1}}%
\!\int}
\def\XXint#1#2#3{{\setbox0=\hbox{$#1{#2#3}{\int}$ }
\vcenter{\hbox{$#2#3$ }}\kern-.6\wd0}}

\def\dashint{\Xint-}

\usepackage[toc,page]{appendix}

\numberwithin{equation}{section}

\begin{document}

\title[Analytic torsion for fibred boundary metrics]
{Analytic torsion for fibred boundary metrics \\
and conic degeneration}

\author{J\o rgen Olsen Lye}
\address{Leibniz Universit\"at Hannover,
30167 Hannover,
Germany}
\email{joergen.lye@math.uni-hannover.de}

\author{Boris Vertman}
\address{Universit\"at Oldenburg,
26129 Oldenburg,
Germany}
\email{boris.vertman@uni-oldenburg.de}

\date{\today}

\begin{abstract}
{We study the renormalized analytic torsion of complete manifolds with fibred boundary metrics,
also referred to as $\phi$-metrics. We establish invariance of the torsion under suitable deformations 
of the metric, and establish a gluing formula. As an application, we 
relate the analytic torsions for complete $\phi$- and incomplete wedge-metrics. 
As a simple consequence we recover a result by Sher and Guillarmou about analytic torsion under conic degeneration.}
\end{abstract}

\maketitle
\tableofcontents

\section{Introduction: basics on analytic torsion}\label{intro-section}

\subsection{Ray-Singer analytic torsion}

Analytic torsion is an important topological invariant, which is defined using spectral data of a compact Riemannian manifold. 
We start by recalling its definition. Let $(M,g)$ be a closed oriented Riemannian manifold of dimension $m$, equipped with a flat Hermitian
vector bundle $(E,\nabla,h)$. Consider the corresponding Hodge Laplacian $\Delta_k$
acting on $E$--valued differential forms $\Omega^k(M,E)$ of degree
$k$. Its unique self-adjoint extension is a discrete operator with spectrum $\sigma(\Delta_k)$
accumulating at $\infty$ according to Weyl's law. As a consequence, we can define its zeta-function as the following convergent sum
\begin{equation}
     \zeta(s,\Delta_k) :=  \sum_{\lambda\in \sigma(\Delta_k) \setminus \{0\}} m(\lambda) \, \lambda^{-s}, \ \Re( s) > \frac{m}{2},
     \label{eq:zeta}
\end{equation}
where $m(\lambda)$ is the geometric multiplicity of the eigenvalue $\lambda\in \sigma(\Delta_k)$.
The zeta-function is linked to the \emph{heat trace}
$\Tr \Bl e^{-t\Delta_k}\Br $ via a Mellin transform
\begin{equation}
\zeta(s,\Delta_k) = \frac{1}{\Gamma(s)} \int_0^{\infty}
    t^{s-1} \Bl\Tr\Bl e^{-t\Delta_k}\Br - \dim \ker \Delta_k \Br \, dt.
    \label{eq:HeatTorsion}
\end{equation}
The short time asymptotic expansion of the
heat trace yields a meromorphic extension of $\zeta(s,\Delta_k)$, which is a priori a holomorphic function 
for $\Re( s) > \frac{m}{2}$,
to the whole complex plane $\C$ with at most simple poles and $s=0$ being a regular point.
This allows us to define the \emph{scalar analytic torsion} of the flat bundle $E$ by
\begin{equation}
T(M,E;g):= \exp\left(\, \frac{1}{2}\sum_{k=0}^{\dim M}(-1)^k\cdot k \cdot 
\zeta'(0,\Delta_k)\right).
\label{eq:AnalyticTorsion}
\end{equation}

\subsection{Analytic torsion as a norm}

Below, it will become convenient to reinterpret the analytic torsion as a norm on the determinant line of the 
cohomology $H^*(M,E)$ with values in the flat vector bundle $E$. This is defined as 
\[\det H^*(M,E):= \bigotimes_{k=0}^m \left( \bigwedge^{b_k} H^k(M,E)\right)^{(-1)^k},\]
where $b_k=\dim H^k(M,E)$. The determinant line bundle has an induced $L^2$ structure $\norm{\cdot }_{L^2}$ 
from the inclusion $H^k(M,E) \cong \ker \Delta_{k} \subset L^2\Omega^k(M,E)$, where the $L^2$-inner product
is induced by the Riemannian metric $g$ and the Hermitian metric $h$. 
The \emph{analytic torsion norm}, also called the Ray-Singer metric or the Quillen metric, is then defined as
\begin{equation}
\norm{\cdot}_{(M,E,g)}^{RS}:= T(M,E;g)\norm{\cdot }_{L^2}.
\label{eq:AnalyticTorsionNorm}
\end{equation}
By an argument of Ray and Singer \cite{RaySin:RTA}, this is actually independent of $g$. 
In case of non-compact or singular manifolds, $H^*(M,E)$ is replaced by the 
reduced $L^2$-cohomology $H^*_{(2)}(M,E)$, which is still isomorphic to the 
space of harmonic forms. Moreover, the Quillen metric, if well-defined, 
may no longer be independent of $g$.

\subsection{Analytic torsion norm on manifolds with (regular) boundary}\label{rel-abs-sec}

Similar construction applies to manifolds $(M,g)$ with regular boundary $\partial M$.
Imposing relative and absolute boundary conditions as in \cite{Hilbert}, one defines
the respective analytic torsion norms 
\begin{equation}
\norm{\cdot}_{(M,E,\partial M, g)}^{RS}, \quad \norm{\cdot}_{(M,E,g)}^{RS}. 
\end{equation}

\subsection{Cheeger-M\"uller theorem}

The analytic torsion $T(M,E;g)$ was introduced by Ray and Singer \cite{RaySin:RTA}
as an analytic analogue of the Reidemeister
torsion $\tau(M,E;\omega)$. The latter is defined in terms of a simplicial or cellular decomposition of $M$, 
where $\omega$ is a choice of a non-zero element in $\det H^*(M,E)$. 
This invariant was introduced by Reidemeister \cite{Rei:UVK}
and Franz \cite{Fra:UDT} and is a homeomorphism-invariant, which distinguishes spaces which are homotopy-equivalent but not homeomorphic.
The central result of this field is due to Cheeger \cite{Che:ATA} and M\"uller \cite{Mue:ATA}.

\begin{thm}[{ \cite{Che:ATA}, \cite{Mue:ATA}}]\label{CM-thm}
Consider a smooth closed Riemannian manifold $(M,g)$ and a flat Hermitian vector bundle $(E,\nabla, h)$. 
Then, when $\norm{\omega }_{L^2} = 1$, the analytic torsion $T(M,E;g)$ equals the combinatorial
Reidemeister torsion $\tau(M,E;\omega)$.  
\end{thm}

For smooth, compact manifolds, several other proofs of the Cheeger-M\"uller theorem have been 
obtained. In particular, there is a proof using microlocal surgery methods by Hassel
\cite{Hass:ASA} as well as one using Witten deformation by Bismut-Zhang
\cite{BisZha:AEO}. M\"uller \cite{Mue-unimodular} generalized the
result to unimodular representations of the fundamental
group, while \cite{BisZha:AEO} are even able to dispense with
the unimodularity assumption. The theorem was extended to manifolds with boundary by L\"uck
\cite{Lue:AAT}, Vishik \cite{Vis:GRS} and Br\"uning-Ma
\cite{BruMa:AAF}. A highly readable recent survey of analytic torsion for compact manifolds is \cite{LottSurvey}. \medskip

\subsection{Cheeger-M\"uller type theorems in singular settings}

Extending Theorem \ref{CM-thm} to non-compact manifolds and to spaces with singularities 
is still partially open, and constitutes a central part of Cheeger's \cite{Che:SGS}
celebrated program of ``extending the theory of the Laplace operator to Riemannian spaces 
with singularities''. The first question is, of course, how to define the two torsions. 
\medskip

For the Reidemeister torsion, Dar \cite{Dar:IRT} has introduced intersection R-torsion using 
the intersection homology groups by Goresky and MacPherson 
\cite{IntersectionHom}. On the analytic side one first has to address if 
$\zeta(s,\Delta_k)$ is well-defined, which is non-trivial since on singular or non-compact 
spaces, the heat operator may not be trace class and heat trace asymptotics need not hold. 
\medskip

These issues were addressed on spaces with fibered cusps by Albin, Rochon and Sher 
in \cite{ARS, ARS2} as well as the second author \cite{Ver}.
For spaces with wedge singularities, analytic torsion has been studied by Mazzeo and 
the second named author \cite{MazVer} and a Cheeger-M\"uller type result is due to Albin, Rochon 
and Sher \cite{ARS3} over an odd dimensional edge $B$, see also \cite{HarVer}.
In the case of asymptotically conical spaces, renormalized analytic torsion has been 
studied by Guillarmou and Sher \cite{GuiShe}. \medskip

We would also like to mention the work of Ludwig on analytic torsion on manifold with conical singularities, \cite{Ludwig20}, \cite{Ludwig22}. 
She uses Witten deformation rather than microlocal analysis and builds on Bismut and Zhang, \cite{BisZha:AEO}. 

\subsection{Gluing formula and reduction to the model case}
   
 Lesch \cite{Les:GFA}, building on the work of Vishik \cite{Vis:GRS}, has proven a gluing 
 formula for the analytic torsion in case of discrete spectrum. The intersection 
 R-torsion also satisfies a gluing property, see \cite{Les:GFA}. This allows one to 
establish a Cheeger-M\"{u}ller type theorem on a singular space by comparing 
the analytic and combinatorial torsions in a singular neighborhood only.

\section{Statement of the main results}

\subsection{Manifolds with fibred boundary and $\phi$-metrics}\label{subsec-phi}

We shall always assume that $M$ is connected. 
To deal with non-compact manifolds $M$, we consider a compactification $\overline{M}$
and view $M$ as its open interior. Assume that the boundary $\partial \overline{M} = \partial M$
is the total space of a fibration 
\[\phi\colon \partial M\to B\]
for some closed manifold $B$ and typical fibre being a closed manifold $F$. 
A $\phi$-metric $g_\phi$ on $M$ is by definition of the following form in a collar neighbourhood $\cU=(0,1)\times \partial M$ 
of the boundary $\partial M$
\begin{equation}
g_\phi = \frac{dx^2}{x^4}+\frac{\phi^*g_{B}}{x^2}+g_{F} + \mathfrak{h} = : g_0+\mathfrak{h},
\label{eq:Phi}
\end{equation}
where $x \in (0,1)$ is the radial function on the collar $\cU$, $g_B$ is a 
Riemannian metric on $B$ and $g_F$ is a symmetric bilinear form on 
$\partial M$, which restricts to a Riemannian metric on the fibres $F$. 
 The higher order term satisfies $\vert \mathfrak{h}\vert_{g_0} =\mathcal{O}(x)$ as $x\to 0$. The triple $(\overline{M},\phi,g_{\phi})$ will be called a \textit{$\phi$-manifold}. We will often suppress part of the structure and refer to $(M,g_{\phi})$ as a $\phi$-manifold.
We also consider a flat Hermitian vector bundle $E$ over $\overline{M}$. \medskip

Consider local coordinates $(x,y,z) \in \cU$, where $y=(y_i)$ is the lift of local coordinates
on the base $B$, and $z=(z_j)$ restrict to local coordinates on the fibres $F$. 
We introduce the $\phi$-tangent bundle over $\overline{M}$ by specifying its smooth sections  
(to be smooth in the interior and) locally near $\partial M$ as
\begin{equation*}
C^\infty(\overline{M}, {}^\phi TM) = C^\infty(\overline{M})\text{-span}\left\langle x^{2}\frac{\partial}{\partial x}, 
x\frac{\partial}{\partial y_{i}}, \frac{\partial}{\partial z_{j}}\right\rangle.
\end{equation*}
Note that the metric $g_\phi$ extends to a smooth positive definite quadratic form on 
${}^\phi TM$ over all of $\overline{M}$.
The sections of the dual bundle ${}^\phi T^*M$, the so-called $\phi$-cotangent bundle, 
are smooth in the interior and locally near $\partial M$ are given by
\begin{equation*}
C^\infty(\overline{M}, {}^\phi T^*M) = 
C^\infty(\overline{M})\text{-span}\left\langle \frac{dx}{x^{2}}, \frac{dy_{i}}{x},dz_{j}\right\rangle.
\end{equation*}

We are building on the works of Grieser, Talebi and the second named author 
\cite{MohammadBorisDaniel, MohammadBoris, Mohammad}.
There, some restrictions are imposed on the geometry, which we now gather together in a single assumption.

\begin{Assump}\label{assumption}
\begin{enumerate}
\item The higher order term satisfies $\vert \mathfrak{h}\vert_{g_0} =\mathcal{O}(x^3)$ as $x\to 0$.
\item The fibration $\phi\colon (\partial M,g_F+\phi^* g_B)\to (B,g_B)$ is a Riemannian submersion.
\item The base manifold $B$ of the fibration $\phi$ has dimension $b=\dim(B)\geq 2$.
\item The twisted Gauss-Bonnet operator $D_B$ on the base $B$ with values in the bundle of fibre-harmonic forms $\mathscr{H}^*(F,E)$ over $B$ 
satisfies some spectral conditions and commutes with the projection onto the fibre-harmonic forms. We refer to 
\cite[Assumption 1.4, 1.5]{MohammadBorisDaniel} for precise statements. 
\end{enumerate}
\end{Assump} 

The simplest example of a $\phi$-manifold is $M = F \times \R^n$ with $g_\phi = g_F + g_{\R^n}$,
where $(F,g_F)$ is a closed Riemannian manifold and $g_{\R^n}$ is the Euclidean metric on $\R^n$. Denote by 
$Y$ the spherical compactification of $\R^n$ at infinity, so that $\overline{M} = F \times Y$ and $\partial M=F \times \mathbb{S}^{n-1}$.
The fibration $\phi: F \times \mathbb{S}^{n-1} \to \mathbb{S}^{n-1}$ is the projection onto the second factor.
If $r>0$ denotes the Euclidean distance to the origin in $\R^n$, then $x=\frac{1}{r}$ is a boundary defining function 
at infinity. The metric induced from the Euclidean metric is precisely 
\[g_\phi=\frac{dx^2}{x^4}+\frac{\phi^*g_{\mathbb{S}^{n-1}}}{x^2} + g_F.\]  
 
This structure might seem contrived, but arises naturally in several interesting examples. 
We refer to \cite[Section 7]{Cohomology} for a more extensive treatment and mention just a few examples here. 
Known $4$-dimensional examples of non-compact, complete hyperk\"{a}hler manifolds
with Riemann tensor in $L^2$ (so-called gravitational instantons), are all $\phi$-manifolds, some of them with non-trivial fibrations. The most well-known class are the ALE spaces, 
like the Eguchi-Hanson space \cite{EH} with trivial fibre and $\partial M= \R\P^3$. The second class of spaces are the ALF-spaces 
(asymptotically locally flat) where $F=\S^1$, like the Gibbons-Hawking spaces \cite{GH} or Taub-NUT. The final family of 
$\phi$-manifolds we would like to mention are the ALG manifolds with 
$F=\S^1\times \S^1$ as constructed in \cite{HKMetrics}. \medskip

\subsection{First main result: invariance of analytic torsion}

In case $\dim F = 0$, the (renormalized) analytic torsion for $\phi$-metrics 
has been studied by Guillarmou and Sher \cite{GuiShe}.
An extension to the case $\dim F \geq 1$ was outlined by Talebi \cite{Mohammad}. In our first main result, we will present
a different construction and establish invariance properties of the analytic torsion. The result is 
is a combination of Definition \ref{eq:LogTDef} and Corollary \ref{Cor:deltaFFormula}.

\begin{thm}\label{main1} Let $(M,g_\phi)$ be a $\phi$-manifold and $(E,h)$ a flat vector bundle over $\overline{M}$.
We impose Assumption \ref{assumption}. Then the following statements hold. 
\begin{enumerate}
\item the renormalized analytic torsion $T(M,E;g_\phi)$ is well-defined; 
\item if $m=\dim(M)$ is odd, then the torsion norm $\norm{\cdot}_{(M,E,g_\phi)}^{RS}$ on the determinant line 
$\det H^*_{2}(M;E)$ of reduced $L^2$-cohomology is invariant under perturbations of $g_\phi$ 
of the form $g_\phi + h$ with 
$\vert h\vert_{g_\phi} =\mathcal{O}(x^{b+1+\alpha})$ as $x\to 0$
for any $\alpha > 0$ and $b = \dim B$. 
\end{enumerate}
\end{thm}

\subsection{Second main result: gluing formula for analytic torsion}

Our second main result is obtained in Theorem \ref{thm:Gluing}, which we summarize as follows.

\begin{thm}\label{main2} Let $(M,g_\phi)$ be a $\phi$-manifold and $(E,h)$ a flat vector bundle over $\overline{M}$.
We impose Assumption \ref{assumption} and \ref{Assumption:GaporAcyc}. 
Assume furthermore that $g_\phi$ is product near $\{1\}\times \partial M \subset M$. Then, 
the analytic torsion admits a gluing formula for a cut along $\{1\}\times \partial M \subset M$, as illustrated in 
Figure \ref{cut-figure}. Namely, consider 
$$
(M_1:= (0,1] \times \partial M, g_\phi \restriction M_1), \quad 
(M_2:= \overline{M \backslash M_1},  g_\phi \restriction M_2).
$$ 
Note that $\partial M_1 = \partial M_2 = \{1\}\times \partial M$.
Then there is a canonical isomorphism of determinant lines 
$$
\Phi: \det H^*_{(2)} (M_1, \partial M_1, E) \otimes \det H^*(M_2, E) \to \det H^*_{(2)} (M,E) ,
$$
and the (renormalized) analytic torsion norms on these determinant lines are related by
\begin{equation}\label{gluing-corr}
\norm{\alpha}_{(M_1, \partial M_1,E;g_\phi)}^{RS} \cdot \norm{\beta}_{(M_2,E;g_\phi)}^{RS} =
2^{\frac{\chi(\partial M,E)}{2}}\norm{\Phi (\alpha \otimes \beta )}_{(M, E; g_\phi)}^{RS}.
\end{equation}
\end{thm}

\begin{figure}[h]
\includegraphics[scale=0.3]{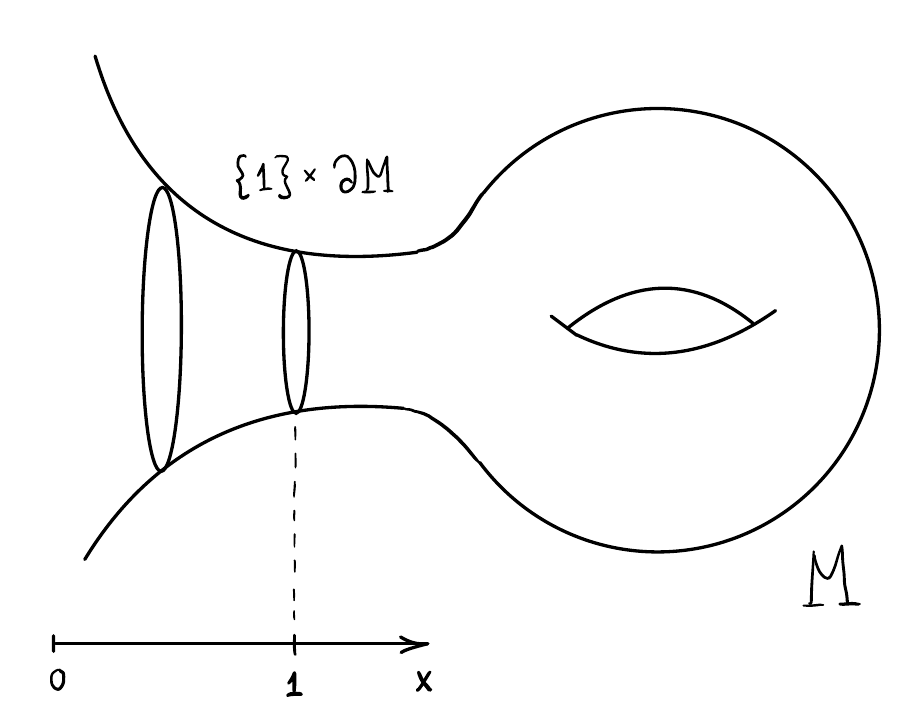}
\caption{Cutting hypersurface $\{1\}\times \partial M \subset M$}
\label{cut-figure}
\end{figure}

Obviously, combining our result with \cite{Les:GFA} yields
a gluing formula for cutting along any separating hypersurface in 
$M$, not necessarily only along $\{1\}\times \partial M$.

\subsection{Third main result: wedge degeneration}

Our third and final main result is an application of the first two results, and is 
concerned with a singular degeneration. We consider a $\phi$-manifold $(M,g_\phi)$
and a family $\cU_\varepsilon = (0,\varepsilon)_x \times \partial M$ of collars of the boundary. 
We set $M_\varepsilon := M \backslash \cU_\varepsilon$. We can now consider a family of closed manifolds $(K_\varepsilon, g_\varepsilon)$
which are obtained by gluing rescaled $\varepsilon M_\varepsilon$ to a compact manifold $K'$ with boundary
$\partial K' = \partial M$. Taking $\varepsilon \to 0$, we obtain a singular manifold $(\Omega,g_\omega)$ in the
limit, with the metric given in terms of $r= 1/x$ in the singular neighborhood near $r=0$ by
\begin{equation}
g_\omega \restriction (0,1)_r \times \partial M = dr^2+ r^2 \phi^*g_{B} + g_F+ \mathfrak{h}(1/r).
\label{eq:wedge}
\end{equation}

\begin{figure}[h]
\includegraphics[scale=0.6]{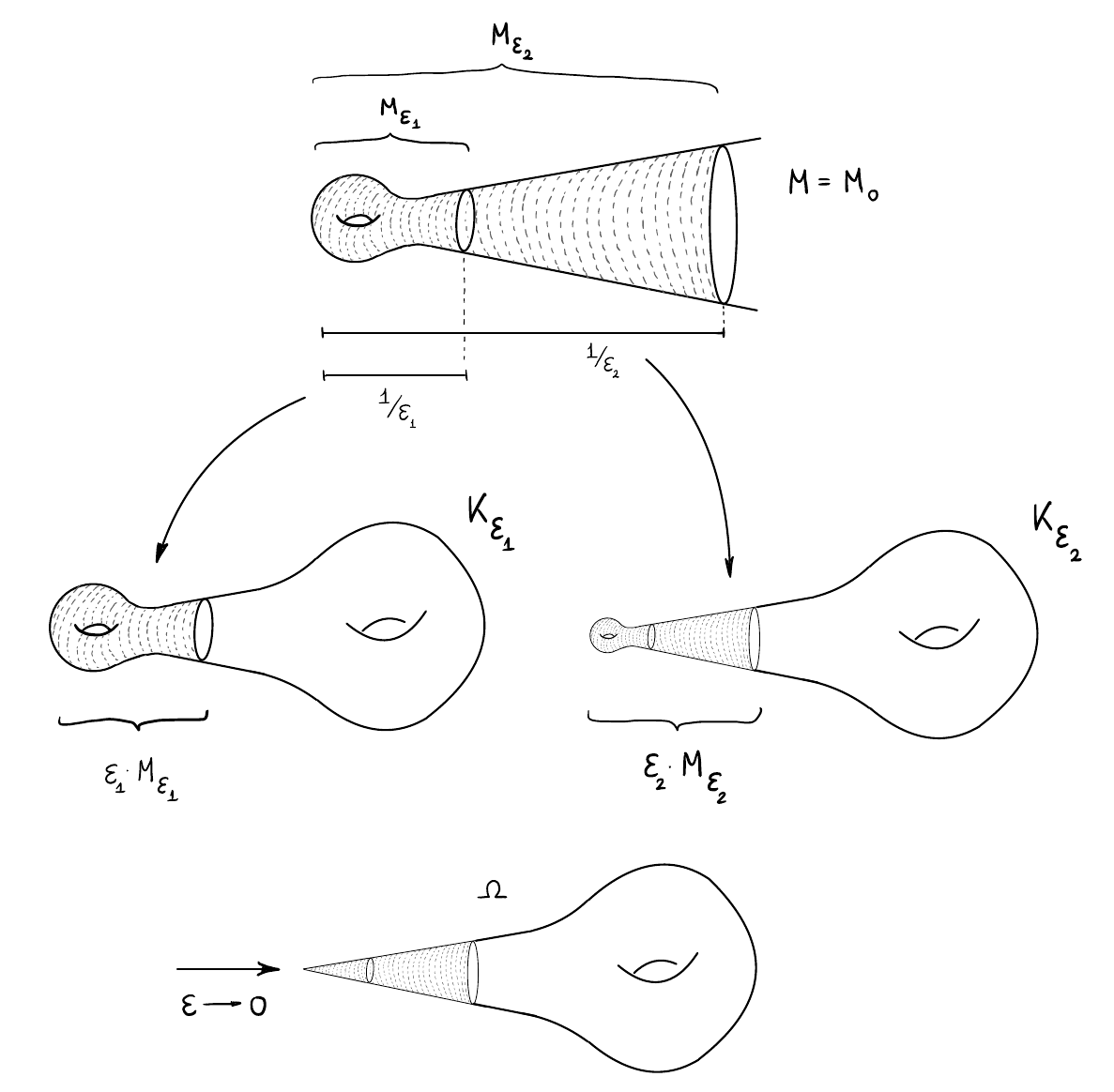}
\caption{Conic and wedge degeneration}
\label{degeneration-figure}
\end{figure}

This is illustrated in Figure \ref{degeneration-figure} for the case $\dim F = 0$,
which has been studied in \cite{GuiShe} under the name of \emph{conic degeneration}. 
The metric $g_\omega$ is not quite a wedge metric, since in the wedge case, $\partial M$
is the total space of a fibration $\psi: \partial M \to F$ with typical fibre $B$ and the 
wedge metrics are of the form
\begin{align}\label{reversed}
dr^2+ r^2 g_{B} + \psi^*g_F,
\end{align}
i.e. the roles of base and fibres being reversed, when compared to \eqref{eq:wedge}.
Hence we can regard the limit space $(\Omega,g_\omega)$ as a wedge space only for 
trivializable fibrations $\phi$. Our third main result is proved below 
in Theorem \ref{main3-rephrazed} and reads as follows (we refer the reader to Theorem \ref{main3-rephrazed} for precise definitions of the spaces $K$ and $\Omega$). 

\begin{thm}\label{main3} Consider the singular degeneration as described above. 
\begin{enumerate}
\item impose  Assumption \ref{assumption} on the complete $\phi$-manifold $(M,g_\phi)$; 
\item write $(K,g):=(K_1, g_1)$ for the compact smooth manifold in Figure \ref{degeneration-figure};
\item assume $\partial M \cong B \times F$ and consider the limiting wedge manifold $(\Omega,g_\omega)$ 
in Figure \ref{degeneration-figure} with wedge along the base $F$ and cone link $B$ (note the interchange of the 
roles of the base and the link). Assume moreover $\dim F$ is even. 
\end{enumerate}
Then there is a canonical isomorphism of determinant lines 
\[
\Phi: \det H^*_{(2)} (M,E) \otimes \det H^*_{(2)} (\Omega, E) \to \det H^* (K,E) ,\]
and the scalar analytic torsions are related for any $\varepsilon > 0$ by
\begin{equation}\label{gluing-corr2}
\begin{split}
\log T(K_\varepsilon, E; g_\varepsilon) &=
\log T(M,E;g_\phi) + \log T(\Omega,E;g_\omega) \\ &+ 
\frac{\norm{\alpha}_{L^2(M,E;g_\phi)} \cdot \norm{\beta}_{L^2(\Omega,E;
g_\omega)}}{\norm{\Phi (\alpha \otimes \beta )}_{L^2(K_\varepsilon, E; g_\varepsilon)}},
\end{split}
\end{equation}
for any non-zero $\alpha \otimes \beta \in \det H^*_{(2)} (M,E) \otimes \det H^*_{(2)} (\Omega, E)$. 
\end{thm}

For $\dim F = 0$, this gives a purely $L^2$-cohomological interpretation of the intricate terms in the second formula of
\cite[Theorem 10]{GuiShe} without the modified spectral Witt condition imposed in \cite[Equation (10)]{GuiShe}.
In other words, this paper provides an independent take on the conic degeneration result in 
 \cite{GuiShe} and relates analytic torsion of a wedge and a $\phi$-manifold as a consequence
 of the gluing property. \bigskip

\noindent \emph{Acknowledgements.} The authors are indebted to Daniel Grieser for offering his invaluable advice concerning 
several aspects of this paper, particular Appendix \ref{appendix-integral}.

\section{Heat kernel for large times on a $\phi$-manifold}

This work relies heavily on the notions of polyhomogenous asymptotic expansions 
and blow-ups. We recall some basic concepts in Appendix \ref{appendix-microlocal}. 
We also refer the reader to \cite{Mel:TAP}, \cite[Chapter 1, Chapter 5]{Mel:Analysis} and 
\cite{Grieser} for much more detailed introductions. 

\subsection{Blow-up for the low energy resolvent kernel}

We shall always abbreviate 
$$
\Lambda^k_\phi M:= \Lambda^k {}^\phi T^*M.
$$
The integral kernel of the resolvent $(\Delta_k+\kappa^2)^{-1},\kappa > 0$ is 
a section (actually a half density) of 
$(\Lambda^k_\phi M \otimes E) \boxtimes (\Lambda^k_\phi M \otimes E)$ pulled back to 
$\overline{M}\times \overline{M}\times (0,\infty)$ by the projection $(p,q,\kappa) \mapsto (p,q)$ onto the first two factors. 
We refer the reader to the careful discussion in \cite{MohammadBorisDaniel} concerning the precise choice of bundles and half-densities and explain here briefly only the blowup of the base manifold 
$\overline{M}\times \overline{M}\times [0,\infty)$, necessary to turn the resolvent into a polyhomogeneous section. \medskip

This blowup is described in detail in \cite[§ 6, § 7]{MohammadBorisDaniel} and is illustrated in
Figure \ref{fig:resolvent-blowup}, with the original space 
$\overline{M}\times \overline{M}\times [0,\infty)$ indicated with thick dotted (blue) 
coordinate axes in the background ($[0,\infty) \equiv \R_+$).

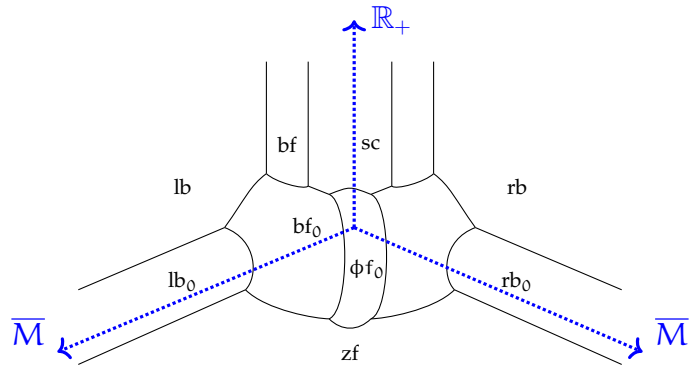
\begin{figure}[h]
\centering
\begin{tikzpicture}[scale =0.55]

\draw[-] (-2,2.3)--(-2,5);
\draw[-] (-1,2)--(-1,5);
\draw[-] (1,2)--(1,5);
\draw[-] (2,2.3)--(2,5);
\draw[-] (-3,1)--(-6.5,-0.5);
\draw[-] (3,1)--(6.5,-0.5);
\draw[-] (-2.5,-0.5)--(-6.5,-2.3);
\draw[-] (2.5,-0.5)--(6.5,-2.3);

\draw (-1,2)--(-0.5,1.8);
\draw (1,2)--(0.5,1.8);
\draw (-0.5,1.8).. controls (0,2) and (0,2) .. (0.5,1.8);

\draw (-1,2).. controls (-1.1,1.9) and (-1.8,2.1) .. (-2,2.3);
\draw (1,2).. controls (1.1,1.9) and (1.8,2.1) .. (2,2.3);

\draw (-2.5,-0.5).. controls (-2,-1) and (-0.8,-1.2) .. (-0.5,-1.2);
\draw (2.5,-0.5).. controls (2,-1) and (0.8,-1.2) .. (0.5,-1.2);
\draw (-0.5,-1.2).. controls (-0.2,-1.5) and (0.2,-1.5) .. (0.5,-1.2);

\draw (-3,1).. controls (-2.2,0.7) and (-2.2,-0.2) .. (-2.5,-0.5);
\draw (3,1).. controls (2.2,0.7) and (2.2,-0.2) .. (2.5,-0.5);
\draw (-2,2.3).. controls (-2.4,2) and (-2.5,1.7) .. (-3,1);
\draw (2,2.3).. controls (2.4,2) and (2.5,1.7) .. (3,1);

\draw (-0.5,-1.2).. controls (0,-1) and (0,1.6) .. (-0.5,1.8);
\draw (0.5,-1.2).. controls (1,-1) and (1,1.6) .. (0.5,1.8);

\node at (-4,-0.3) {\tiny{$\textup{lb}_{0}$}};
\node at (4,-0.3) {\tiny{$\textup{rb}_{0}$}};
\node at (-4,2) {\tiny{$\textup{lb}$}};
\node at (4,2) {\tiny{$\textup{rb}$}};
\node at (0,-2) {\tiny{$\textup{zf}$}};
\node at (-1,1) {\tiny{$\textup{bf}_{0}$}};
\node at (0.41,0) {\tiny{$\phi f_0$}};
\node at (0.5,3) {\tiny{$\textup{sc}$}};
\node at (-1.5,3) {\tiny{$\textup{bf}$}};

\draw[->, very thick, densely dotted, blue] (0.1,1)--(0.1,6);
\node[very thick, blue] at (1,6) {$\R_+$};
\draw[->, very thick, densely dotted, blue] (0.1,1)--(-7,-2);
\node[very thick, blue] at (-7.7,-1.5) {$\overline{M}$};
\draw[->, very thick, densely dotted, blue] (0.1,1)--(7,-2);
\node[very thick, blue] at (7.7,-1.5) {$\overline{M}$};

\end{tikzpicture}
 \caption{Resolvent blowup space $M^{2}_{\kappa, \phi}$}
  \label{fig:resolvent-blowup}
\end{figure}

The blowup is obtained as follows. First, one blows up 
the codimension $3$ corner $\partial M \times \partial M \times \{0\}$,
which defines a new boundary hypersurface $\textup{bf}_{0}$.
Then one blows up the codimension $2$ corners
$\overline{M} \times \partial M \times \{0\}$, $\partial M \times \overline{M} \times \{0\}$
and $\partial M \times \partial M \times \R_+$, which define new 
boundary faces $\textup{lb}_{0}$, $\textup{rb}_{0}$ and $\textup{bf}$, respectively. 
Next we blow up the (lifted) interior fibre diagonal 
\begin{align*}
&\textup{diag}_{\phi,\mathrm{int}} \times \R_+ = \{ (p, q, \kappa) \in \overline{\cU} \times 
\overline{\cU} \times \R_+: \phi(p) = \phi(p')\},
 \end{align*}
intersected with $\textup{bf}$. Here $\cU$ is a collar neighbourhood (so that $\phi$ is defined). This defines a new boundary hypersurface $\textup{sc}$. Finally, 
the resolvent blowup space $M^{2}_{\kappa, \phi}$ is obtained by one last blowup
of the (lifted) interior fibre diagonal with $\textup{bf}_{0}$, which defines the
boundary hypersurface $\phi f_0$. The resolvent blowup space comes with a 
canonical blowdown map
\[
\beta_{\kappa,\phi}\colon M_{\kappa,\phi}^2\to \overline{M}\times \overline{M}\times [0,\infty).
\]
The resolvent lifts to a polyhomogeneous section on the blowup space
$M_{\kappa,\phi}^2$ with a conormal singularity along the lifted diagonal 
in the sense of the result below. We refer the reader to 
\cite[Definition 7.4]{MohammadBorisDaniel} for a precise definition of the 
split calculus $(\kappa,\phi)$-calculus $\Psi^{-2,\mathcal{E}}_{\kappa,\phi,\mathscr{H}}(M)$
and can now state the following theorem.

\begin{thm}[{\cite[Theorem 7.11, Theorem 8.1]{MohammadBorisDaniel}}] \label{thm:IndexSets1}
Let $\Box_k=x^{-\frac{b-1}{2}} \Delta_k x^{\frac{b-1}{2}}$. 

\begin{enumerate}
\item Then $(\Box+\kappa^2)^{-1}$ is an element of the split $(\kappa,\phi)$-calculus $\Psi^{-2,\mathcal{E}}_{\kappa,\phi,\mathscr{H}}(M)$ with 
\begin{equation*}
\mathcal{E}_{\textup{sc}}\geq 0, \quad \mathcal{E}_{\phi \textup{f}_0}\geq 0, \quad \mathcal{E}_{bf_0}\geq -2, \quad \mathcal{E}_{lb_0},\mathcal{E}_{rb_0}>0, \quad \mathcal{E}_{\textup{zf}}\geq -2.
\end{equation*}

\item For any integer $\sigma$, $(\Box_k+\kappa^2)^{-\sigma}$ lies in 
the split  $(\kappa,\phi)$-calculus $\Psi^{-2\sigma,\mathcal{F}}_{\kappa,\phi,\mathscr{H}}(M)$ with
\begin{align*}
\mathcal{F}_{\textup{sc}}\geq 0, \hspace{1cm} \mathcal{F}_{\phi \textup{f}_0}\geq 
\min \, \{ \, 0, -2\sigma + (b+1) \, \}, \hspace{1cm} \mathcal{F}_{bf_0}\geq -2\sigma, 
\notag \\ \mathcal{F}_{lb_0},\mathcal{F}_{rb_0}>-2(\sigma-1), \hspace{1cm} 
\mathcal{F}_{\textup{zf}}\geq -2\sigma.
\end{align*}
\end{enumerate}
\end{thm}

\begin{proof}
The first statement is simply \cite[Theorem 7.11]{MohammadBorisDaniel}. 
The second statement follows from the first statement and the composition result
\cite[Theorem 8.1]{MohammadBorisDaniel} by inductive
composition of the resolvent with itself. 
\end{proof}

The conjugation by $x^{\frac{b-1}{2}}$ in going from $\Delta_k$ to $\Box_k$ changes the 
index sets away from the diagonal, but not along the (blown up) diagonal, i.e. the index sets for 
$\textup{zf}$, $\textup{sc}$, and $\phi \textup{f}_0$ are the same for $\Delta_k$ and $\Box_k$. 
These are the only index sets we will need later in our analysis.

\begin{Rem}\label{complex-k}
The above resolvents are for $\kappa$ real. We will need complex $\kappa$ as well below, but as in Remark \cite[Remark 20]{GuiShe}, the resolvent $(\Box_k+\kappa^2e^{2i\theta})^{-1}$ behaves as $ (\Box_k+\kappa^2)^{-1}$ for any $\theta\in \left(-\frac{\pi}{2},\frac{\pi}{2}\right)$ and the resolvents have a smooth $\theta$ dependence. This follows from the parametrix construction in \cite{MohammadBorisDaniel}. 
\end{Rem}

\subsection{The diagonal in $M^2_{\kappa,\phi}$}

We write $\textup{diag}_{\kappa,\phi}$ for the closure of the diagonal in 
the resolvent blowup space $M^2_{\kappa,\phi}$, namely 
\[
\textup{diag}_{\kappa,\phi}\coloneqq  \beta_{\kappa,\phi}^{-1}
(\textup{diag}(\overline{M})\times [0,\infty)_\kappa).
\]

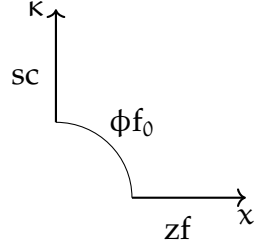
\begin{figure}[h]
\centering
\begin{tikzpicture}
\node at (-0.4,1.6) {$\textup{sc}$};
\node at (1.6,-0.4) {$\textup{zf}$};
\node at (1,1) {$\phi\textup{f}_0$};
\draw[thick,->] (1,0) -- (2.5,0) node[anchor=north] {$x$};;
\draw[thick,->] (0,1) -- (0,2.5) node[anchor=east] {$\kappa$};;
\draw (1,0) arc (0:90:1cm);
\end{tikzpicture}
\caption{The diagonal $\textup{diag}_{\kappa,\phi} \subset M^2_{\kappa,\phi}$.}
\label{lifted-diagonal}
\end{figure}

This is a $p$-submanifold of $M^2_{\kappa,\phi}$, and, as is visible in the
illustration of $M^2_{\kappa,\phi}$ in Figure \ref{fig:resolvent-blowup}, the only faces it intersects are 
$sc$, $\phi f_0$ and $zf$. We will abuse notation and write $sc, \phi f_0$ and $zf$ also for these intersections, as in 
Figure \ref{lifted-diagonal}. 

\subsection{The pointwise trace of the heat operator for large times}

In order to pass from the resolvent to the 
heat operator, recall that the resolvent and the heat operator are 
related by
\begin{equation}
e^{-t\Delta_k}=\frac{1}{2\pi i}\int_{\Gamma} e^{\xi t} (\Delta_k+\xi)^{-1}\, d\xi,
\label{eq:HeatOperator}
\end{equation}
where the contour $\Gamma = \Gamma(\theta,t)$ depends on parameters
$\theta\in \left(\frac{\pi}{2},\pi\right)$ and the time $t>0$ of the heat operator in \eqref{eq:HeatOperator}.
It is chosen to avoid the spectrum of $-\Delta_k$ and is illustrated in Figure \ref{fig:Contour}.
By construction, it consists of the 3 components 
\begin{equation}\label{gamma-components}
\ell_{\pm \theta} :=\{e^{\pm i\theta}\tau\, \vert \, \tau\geq t^{-1}\}, \quad
C_{\theta}:= \{t^{-1}e^{i\psi}\, \vert\, \psi\in (\theta,-\theta)\}.
\end{equation} 

\begin{figure}[h]
\begin{tikzpicture}[scale=0.65]
\draw[thick,->-=.5] (-1,-1)--(-3,-3);
\draw[thick,->-=.5] (-1,1) arc (135:-135:1.41);
\draw[thick,->-=.5] (-3,3)--(-1,1);
\draw[->] (-4,0)--(4,0);
\draw[->] (0,-4)--(0,4);
\filldraw[black] (0,0) circle (1pt) node[anchor=west] {$0$};
\node[anchor=east] at (2.5,2.5) {$\Gamma$};
\node[anchor=east] at (-2.5,2) {$\ell_{+ \theta}$};
\node[anchor=east] at (-2.5,-2) {$\ell_{- \theta}$};
\node[anchor=east] at (2,-1.5) {$C_{\theta}$};
\end{tikzpicture}
\caption{The contour $\Gamma$ in the Dunford integral \eqref{eq:HeatOperator}.}
\label{fig:Contour}
\end{figure}
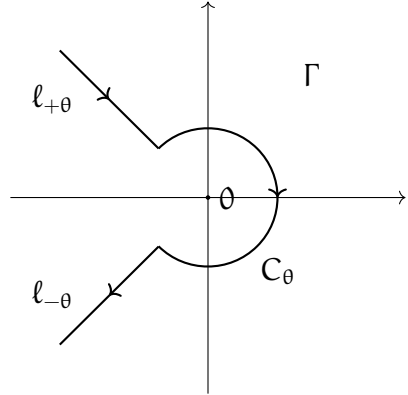
A relation similar to \eqref{eq:HeatOperator} is also valid between the integral kernels of 
$e^{-t\Delta_k}$ and a sufficiently high power of $(\Delta_k+\xi)^{-1}$.

\begin{Prop}\label{Prop:PolyBound}
For any $\nu\in \N$ we have a relation between operators
\begin{equation}
e^{-t\Delta_k}=\frac{(\nu-1)!}{(-t)^{(\nu-1)}}\frac{1}{2\pi i} \int_{\Gamma}  e^{\xi t} (\Delta_k+\xi)^{-\nu}\, d\xi,
\label{eq:ModifiedHeat}
\end{equation} 
where $\Gamma$ avoids the spectrum of $-\Delta_k$ as defined above. Choosing
\[ 
\nu=\left[\frac{m}{2}\right]+1=\left[\frac{\dim M}{2}\right]+1,
\]
\eqref{eq:ModifiedHeat} is valid for the integral kernels $H_k$ of $e^{-t\Delta_k}$ and 
$K_R$ of $(\Delta_k+\xi)^{-\nu}$. Moreover
\[
\norm{K_R}_{C^0(M\times M)}\leq C  \left\vert \textup{Im}(\xi) \right\vert^{-\nu} (1+2\vert \xi\vert)^{\nu},
\]
for some constant $C>0$, uniformly in $\| \xi \|\geq \delta$
for any fixed $\delta> 0$. Same statements hold for any $M$ of bounded geometry.
\end{Prop}  

\begin{proof}
Iterative integration by parts in \eqref{eq:HeatOperator} yields the formula \eqref{eq:ModifiedHeat}. 
In order to prove that the relation holds on the level of integral kernels, we 
show continuity of $K_R$ across the diagonal. The argument actually holds for the wider class of 
manifolds of bounded geometry. \cite[Proposition 3.5]{Vaillant}, building upon 
\cite[Theorem 3.7]{ShubinNonCompact} asserts for $\nu=\left[\frac{m}{2}\right]+1$
\[
\norm{K_R}_{C^0(M\times M)}\leq C \norm{(\Delta_k+\xi)^{-\nu}}_{\mathcal{B}(H^{-\nu},H^{\nu})},
\]
where $H^{-\nu} \equiv H^{-\nu}(M,E_k)$ and $H^{\nu} \equiv H^\nu (M,E_k)$ 
are the usual Sobolev spaces and $E_k\coloneqq E\otimes \Omega^k$. The norm on the right hand side is estimated by a little trick. 
Consider the following commutative diagram
\begin{center}
\begin{tikzcd}[row sep=large, column sep=large]
H^{-\nu}(M,E_k) \arrow{d}[left]{(\Delta_k+1)^{-\nu/2}} \arrow{r}{(\Delta_k+\xi)^{-\nu}}
& H^{\nu}(M,E_k) \\
L^2(M,E_k) \arrow{r}{P}
& L^2(M,E_k) \arrow{u}[right]{(\Delta_k+1)^{-\nu/2}}
\end{tikzcd}
\end{center}
where $P$ is defined to make the diagram commute; 
\begin{equation}\label{eq:P-Expr}
P\coloneqq (\Delta_k+1)^{\nu/2}\circ (\Delta_k+\xi)^{-\nu}\circ (\Delta_k+1)^{\nu/2}
=\sum_{j=0}^{\nu} \binom{\nu}{j}(1-\xi)^j (\Delta+\xi)^{-j}.
\end{equation}
We have $\xi$-independent bounds  from above (after possibly enlarging $C>0$)
\[\norm{(\Delta_k+1)^{-\nu/2}}_{\mathcal{B}(H^{-\nu},L^2)}\leq C,\]
\[\norm{(\Delta_k+1)^{-\nu/2}}_{\mathcal{B}(L^2, H^{\nu/2})}\leq C.\]
By \eqref{eq:P-Expr} we conclude
\begin{align*}
\norm{(\Delta_k+\xi)^{-\nu}}_{\mathcal{B}(H^{-\nu}, H^{\nu})}
&\leq C^2 \norm{P}_{\mathcal{B}(L^2, L^2)}, \\
&\leq C^2 \sum_{j=0}^{\nu} \binom{\nu}{j}
\vert 1-\xi\vert^j \norm{(\Delta_k+\xi)^{-j}}_{\mathcal{B}(L^2,L^2)}.
\end{align*}
By the well-known estimate \cite[Example 4, p. 210]{Yosida},
\[
\norm{(\Delta_k+\xi)^{-j}}_{\mathcal{B}(L^2,L^2)}\leq \left\vert \textup{Im}(\xi) \right\vert^{-j},
\]
we finally conclude
\begin{align*}
\norm{K_R}_{C^0(M\times M)} &\leq C \norm{(\Delta_k+\xi)^{-\nu}}_{\mathcal{B}(H^{-\nu}, H^{\nu})} 
\\ &\leq C^3\left\vert \textup{Im}(\xi) \right\vert^{-\nu}\sum_{j=0}^{\nu} 
\binom{\nu}{j}\vert 1-\xi\vert^j \left\vert \textup{Im}(\xi) \right\vert^{\nu-j}
\\ &\leq C^3 \left\vert \textup{Im}(\xi) \right\vert^{-\nu} (1+2\vert \xi\vert)^{\nu}.
\end{align*}
\end{proof}

Let us finally write down the asymptotic description for the pointwise trace of the 
resolvent kernel. This is a straightforward consequence of Theorem \ref{thm:IndexSets1}.
From now on, we  shall fix $\nu$ as in Proposition \ref{Prop:PolyBound};
\[\nu=\left[\frac{m}{2}\right]+1=\left[\frac{\dim M}{2}\right]+1.\]
\begin{Lem}
\label{lem:ResolventTrace}
The pointwise trace satisfies $\tr (\Delta_k+\kappa^2)^{-\nu} = F \cdot d\Vol_{\phi}$,
where $\beta_{\kappa,\phi}^* F$ is a polyhomogenous function 
on $\diag_{\kappa,\phi}$ with index set bounds (in notation of Theorem \ref{thm:IndexSets1})
\begin{align*}
\mathcal{F}_{\textup{sc}} \geq 0, \quad
\mathcal{F}_{\textup{zf}} \geq -2\nu, \quad
\mathcal{F}_{\phi \textup{f}_0} \geq -2\nu + (b+1).
\end{align*}
\end{Lem}
\begin{proof}
The index sets in Theorem \ref{thm:IndexSets1} are defined for the Schwartz kernels as 
sections of the half-density bundle $\Omega^{1/2}_{b\phi}(M^2_{\kappa,\phi})\otimes \End(E)$, see e.g. 
\cite[Definition 7.3]{MohammadBorisDaniel}. This latter half-density bundle is defined 
in \cite[(7.7)]{MohammadBorisDaniel} as
\[\Omega^{1/2}_{b\phi}\left(M^2_{\kappa,\phi}\right):= \rho_{\textup{sc}}^{-(b+1)/2}\rho_{\phi \textup{f}_0}^{-(b+1)/2} \Omega_b^{1/2}\left(M_{\kappa,\phi}^2\right).\]
Here, $\Omega_b \left(M_{\kappa,\phi}^2\right)$ denotes the space of the so-called $b$-densities, 
obtained by dividing smooth densities on $M_{\kappa,\phi}^2$ by the product of defining functions
of the boundary hypersurfaces. $\Omega_b^{1/2}\left(M_{\kappa,\phi}^2\right)$ are the corresponding half-densities. 
\medskip

Consider a smooth density $\omega$ on $\overline{M}^2 \times [0,\infty)_\kappa$ and boundary 
defining functions $x$ and $x'$ on the two copies of $\overline{M}$, extending the $x$-coordinates on the collar neighbourhoods. The $b$-density $\omega_b:= (\kappa xx')^{-1} \omega$ lifts to $M^2_{\kappa,\phi}$ as follows 
\begin{align}\label{projective1}
\beta_{\kappa,\phi}^* \omega_b \in (\rho_{\textup{sc}} \rho_{\phi \textup{f}_0})^{b+1} 
\Omega_b \left(M_{\kappa,\phi}^2\right).
\end{align}
Here, the defining function $\rho_{\textup{sc}}$ of $\textup{sc}$ comes with the power $(b+1), b= \dim B$, since 
the boundary hypersurface $\textup{sc}$ arises by blowing up a submanifold of $\textup{bf}$ of codimension $(b+1)$. 
Similarly, the defining function $\rho_{\phi \textup{f}_0}$ of $\phi \textup{f}_0$ comes also with the power $(b+1)$, since 
the boundary hypersurface $\phi \textup{f}_0$ arises by blowing up a submanifold of $\textup{bf}_0$ of codimension $(b+1)$.
Note furthermore 
\begin{align}\label{projective2}
\frac{d\kappa}{\kappa} \, d\Vol_{\phi}(p) \, \, d\Vol_{\phi}(q) = (xx')^{-1-b} \omega_b,
\end{align}
up to multiplication by some smooth function on $\overline{M}^2 \times [0,\infty)$.
Making only the behaviour near $\textup{sc}$, $\phi \textup{f}_0$ and $\textup{zf}$ explicit, we obtain
\begin{align}\label{projective3}
\beta_{\kappa,\phi}^* (xx') = 
(\rho_{\textup{sc}} \rho_{\phi \textup{f}_0})^{2}.
\end{align}
Combining \eqref{projective1}, \eqref{projective2} and \eqref{projective3} yields
\begin{align*}
\beta_{\kappa,\phi}^* \Bigl( \frac{d\kappa}{\kappa} \, d\Vol_{\phi}(p) \, \, d\Vol_{\phi}(q) \Bigr)  
\in (\rho_{\textup{sc}} \rho_{\phi \textup{f}_0})^{-1-b} 
\Omega_b \left(M_{\kappa,\phi}^2\right),
\end{align*}
where we again discarded contributions at the other boundary faces, 
except $\textup{sc}$, $\phi \textup{f}_0$ and $\textup{zf}$.
Consequently
\begin{align*}
&\mu\coloneqq \beta_{\kappa,\phi}^* \Bigl( \frac{d\kappa}{\kappa} \, d\Vol_{\phi}(p) \, \, d\Vol_{\phi}(q) \Bigr)^{\frac{1}{2}}  
\\ &\in (\rho_{\textup{sc}}\rho_{\phi \textup{f}_0})^{-\frac{b+1}{2}} \Omega_b^{1/2}\left(M_{\kappa,\phi}^2\right)
= \Omega^{1/2}_{b\phi}\left(M^2_{\kappa,\phi}\right).
\end{align*}
Hence, writing $(\Delta_k+\kappa^2)^{-\nu}$ with respect to the half density $\mu$ instead of 
$\Omega^{1/2}_{b\phi}(M^2_{\kappa,\phi})\otimes \End(E)$ does not change the asymptotics
at $\textup{sc}$, $\phi \textup{f}_0$ and $\textup{zf}$. Note that $\diag_{\kappa,\phi}$ does not intersect 
any other hypersurfaces. The statement now follows directly from Theorem \ref{thm:IndexSets1}, where we now discard
the component $\sqrt{d\kappa / \kappa}$ in the density. 
\end{proof}

\begin{thm}
\label{thm:LongHeatKernelPoly}
The pointwise trace satisfies $\tr(e^{-t\Delta_k}) = F' \cdot d\Vol_{\phi}$
where $F'$ lifts to a polyhomogenous function on 
$\textup{diag}_{\kappa,\phi}$ for $\kappa^2 = t^{-1} \leq 1$
with index set bounds 
$$\mathcal{E}'_{\textup{sc}}\geq 0, \quad
\mathcal{E}'_{\textup{zf}}\geq 0, \quad
\mathcal{E}'_{\phi \textup{f}_0}\geq b+1,$$
at $\textup{sc}$, $\textup{zf}$ and $\phi \textup{f}_0$, respectively.
\end{thm}

\begin{proof}
We start with \eqref{eq:ModifiedHeat}. Let us drop $k$ from the notation for the rest of the 
proof and write $H$ for the heat kernel and $K_R$ for the integral kernel of 
$(\Delta_k+\xi)^{-\nu}$, both with respect to the $\phi$-volume form $\dVol_{\phi}$. 
With this notation, we have at any $(p,t) \in M \times (0,\infty)$
\begin{equation}\label{trH}
\tr(H)(p,t)=\frac{(\nu-1)!}{(-t)^{(\nu-1)}}\frac{1}{2\pi i} \int_{\Gamma}  
e^{\xi t} \tr(K_R)(p,\xi)\, d\xi.
\end{equation}
We will sometimes drop the reference to the point $p$ in the notation below.
The contour $\Gamma = \Gamma(\theta,t)$ depends on parameters
$\theta\in \left(\frac{\pi}{2},\pi\right)$ and the time $(0,\infty)$, and has already been
introduced above right after \eqref{eq:HeatOperator}. Recall the partition
of the contour in \eqref{gamma-components}, which we now partition further
(we assume $t\geq 1$ here) 
\begin{align*}
\ell^{0}_{\pm \theta} := \{e^{\pm i\theta}\tau\, \vert \, 1/t \leq \tau\leq 1\}, \quad
\ell^{\infty}_{\pm \theta} := \{e^{\pm i\theta}\tau\, \vert \, \tau\geq 1\}, \quad
C_{\theta} := \{t^{-1}e^{i\psi}\, \vert\, \psi\in (\theta,-\theta)\}. 
\end{align*}

\subsubsection*{The integrals along the rays $\ell^{\infty}_{\pm \theta}$:}

We discuss the integral over $\ell^{\infty}_{+\theta}$, 
the argument for $\ell^{\infty}_{-\theta}$ works in exactly the same way, with the same bounds.
We parametrize the contour $\ell^{\infty}_{+\theta}$ by $\xi=\tau e^{i\theta}$ with $\tau \in [1,\infty)$ and find
\[
\int_{\ell^{\infty}_{+\theta}} e^{\xi t} \tr(K_R)(p,\xi)\, d\xi 
= - e^{i\theta}\int_{1}^\infty e^{t \tau \cos\theta} e^{it \tau \sin\theta} \tr(K_R)(p, \tau e^{i\theta})\, d\tau.
\]
Abbreviating $c_\theta=-\cos(\theta)>0$, we use Proposition \ref{Prop:PolyBound} to find a uniform bound 
$|\tr(K_R)(p,\tau e^{i\theta})| \leq C$ for $\tau \in [1,\infty)$, and consequently a bound on the integral
\[
\left\vert \int_{\ell^{\infty}_{+\theta}} e^{\xi t} \tr(K_R)(p,\xi)\, 
d\xi\right \vert \leq C\int_{1}^{\infty} e^{-t\tau c_{\theta}}  d\tau=\frac{C}{c_{\theta}} e^{-c_{\theta}t}.
\]
Hence the part of the contour integral on the right hand side of \eqref{eq:ModifiedHeat}, 
along $\ell^{\infty}_{+\theta}$ is exponentially decaying as $t\to \infty$, uniformly in $p \in M$. 
Hence this component vanishes to infinite order at $\textup{zf}$, $\phi \textup{f}_0$ and has
the index set $\mathcal{F}_{\textup{sc}}$ in its asymptotics at $\textup{sc}$.

\subsubsection*{The integral along the arc $C_{\theta}$:}

We parametrize $\xi=\frac{1}{t} e^{i\psi}$ and find
\begin{align*}
\frac{(\nu-1)!}{(-t)^{(\nu-1)}}\frac{1}{2\pi i} \int_{C_\theta} e^{\xi t} \tr(K_R)(p,\xi)\, d\xi
= - \frac{(\nu-1)!}{2\pi (-t)^{\nu}} \int_{-\theta}^\theta e^{i\psi + e^{i\psi}}\tr(K_R)\left(p, \frac{1}{t} e^{i\psi}\right)
d\psi.
\end{align*}
Since the integrand is polyhomogeneous in $[\overline{M} \times [0,\infty)_{1/ \sqrt{t}}; \partial M \times \{0\}]$,
uniformly in $\psi$, the same is true after integration with the same index sets, up to a shift due to 
the factor of $t^{-\nu}$, which lifts by $\beta_{\kappa,\phi}$ (we set $\kappa := 1/\sqrt{t})$) to $\rho^{2\nu}_{\phi \textup{f}_0}\rho^{2\nu}_{\textup{zf}}$. Hence, using Remark \ref{complex-k}, we
find following index sets for this component 
\begin{align*}
\mathcal{F}_{\textup{sc}} \geq 0, \quad
\mathcal{F}_{\textup{zf}} + 2\nu \geq 0, \quad
\mathcal{F}_{\phi \textup{f}_0} +2\nu \geq b+1,
\end{align*}
for $\textup{sc}$, $\textup{zf}$ and $\phi \textup{f}_0$, respectively.

\subsubsection*{The integrals along the segments $\ell^0_{\pm \theta}$:}

We discuss the integral over $\ell^{0}_{+\theta}$, 
the argument for $\ell^{0}_{-\theta}$ works in exactly the same way, with the same index sets.
We parametrize $\xi=\tau e^{i\theta}$ and find
\begin{equation}\label{third-integral}\begin{split}
\int_{\ell^{0}_{+\theta}} e^{\xi t}\tr( K_R)(p,\xi)\, d\xi &=-
e^{i\theta}\int_{1/t}^1  e^{t \tau e^{i\theta}} \tr(K_R)(p, \tau e^{i\theta})\, d\tau \\
&= -e^{i\theta}\int_{1/t}^1  e^{-t \tau c_\theta} e^{it \tau \sin\theta} \tr(K_R)(p, \tau e^{i\theta})\, d\tau \\
&= -e^{i\theta}\int_{1/\sqrt{t}}^1  e^{-t s^2 c_\theta} e^{it s^2 \sin\theta} \tr(K_R)(p, s^2 e^{i\theta})\, 2s^2 \, \frac{ds}{s},
\end{split}\end{equation}
where we have abbreviated $c_\theta=-\cos(\theta)>0$ as above and substituted $s= \sqrt{\tau}$ in the last step.
To see that this is polyhomogeneous, we can either derive asymptotic expansions of the 
integrals “by hand“. Alternatively, we can argue elegantly using a $b$-triple space argument
in Proposition \ref{auxiliary-pushforward}, where the argument remains unchanged if $[0,\infty)_{x_1}$
is replaced by $\overline{M}$. Thus the asymptotics of \eqref{third-integral} is given by the following index sets
\begin{align*}
\mathcal{F}_{\textup{sc}} \geq 0, \quad
\mathcal{F}_{\textup{zf}} + 2 \geq - 2\nu + 2, \quad
\Bigl( \mathcal{F}_{\phi \textup{f}_0} + 2\Bigr) \, \overline{\cup} \, \mathcal{F}_{\textup{sc}}  \geq - 2\nu + b+3,
\end{align*}
for $\textup{sc}$, $\textup{zf}$ and $\phi \textup{f}_0$, respectively.
Here, the $2$ on the left hand side of the inequalities is due to the $s^2$ factor in \eqref{third-integral}.
Thus, taking into account the shift by $t^{-\nu+1}$ in the formula \eqref{trH}, which lifts to 
$\rho^{2\nu-2}_{\phi \textup{f}_0}\rho^{2\nu-2}_{\textup{zf}}$, we find following index sets 
for the component of $\tr(H)$ over $\ell^{0}_{+\theta}$
\begin{align*}
\mathcal{F}_{\textup{sc}} \geq 0, \quad
\mathcal{F}_{\textup{zf}} + 2\nu \geq 0, \quad
\Bigl( \mathcal{F}_{\phi \textup{f}_0} + 2\nu \Bigr) \, \overline{\cup} \, \mathcal{F}_{\textup{sc}}  \geq b+1,
\end{align*}
for $\textup{sc}$, $\textup{zf}$ and $\phi \textup{f}_0$, respectively.
This concludes the proof of Theorem \ref{thm:LongHeatKernelPoly} by taking minima of all the 
contributions above at the individual faces.
\end{proof}

\begin{Rem}
For $\phi$-metrics with trivial fibres, $\dim F=0$, \cite{She} proves Theorem 
\ref{thm:LongHeatKernelPoly} in more generality in that they describe 
the heat kernel beyond the diagonal. The methods here can be easily extended
to get a description of the heat kernel the full $M^2_{\kappa,\phi}$ with $\kappa = 1/\sqrt{t}$.
\end{Rem}

From now on we will fix the natural $\phi$-volume $d\Vol_{\phi}$
as the measure of integration and identify the pointwise trace 
$\tr(e^{-t\Delta_k}) = F' \cdot d\Vol_{\phi}$ with the 
scalar function $F'$, satisfying the asymptotics in 
Theorem \ref{thm:LongHeatKernelPoly}.

\section{Heat kernel for small times on a $\phi$-manifold}
From now on, we shall write  $e^{-t\Delta_k}$  for both the heat kernel (previously called $H$ or $H_k$ and the operator. The heat kernel $e^{-t\Delta_k}$  is a section of 
$(\Lambda^k_\phi M \otimes E) \boxtimes (\Lambda^k_\phi M \otimes E)$ pulled back to 
$\overline{M}\times \overline{M}\times (0,\infty)$ as above. Here we briefly explain the blowup of the base manifold 
$\overline{M}\times \overline{M}\times [0,\infty)$, necessary to turn the heat kernel into a polyhomogeneous section
as $t \to 0$. This blowup is described in detail in \cite[§ 5]{MohammadBoris} and is illustrated in
Figure \ref{fig:heat-blowup}, with the original 
$\overline{M}\times \overline{M}\times [0,\infty)$ indicated with thick dotted (blue) coordinate axes in the background. \medskip

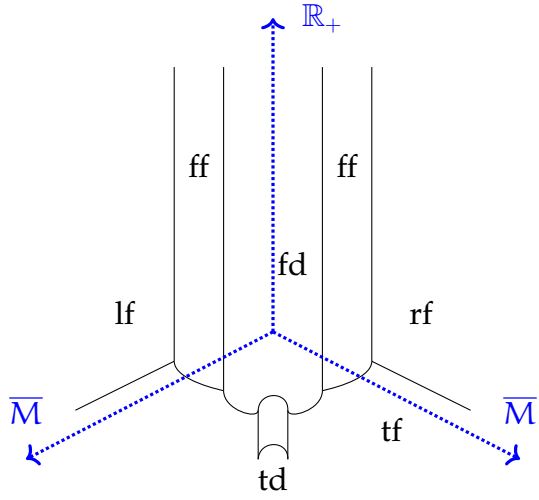
\begin{figure}[h]
\centering

\begin{tikzpicture}[scale = 1.3]

     \draw[-](1,-1) --(2,-1.5);
     \draw[-](-1,-1)--(-2,-1.5);
      \draw[-](1,-1) --(1,2);
      \draw[-](-1,-1)--(-1,2);
      \draw[-](-0.5,-1.3)--(-0.5,2);
      \draw[-](0.5,-1.3)--(0.5,2);
      
          \node at (0.2,0) {$\textup{fd}$};
         \node at (0.75,1) {$\textup{ff}$};
         \node at (-0.75,1) {$\textup{ff}$};
           \node at (-1.5,-0.5) {$\textup{lf}$};
          \node at (1.5,-0.5) {$\textup{rf}$};
           \node at (1.2,-1.7) {$\textup{tf}$};
           
         \draw (1,-1).. controls (1,-1.2) and (0.5,-1.3)..(0.5,-1.3);
  
            \draw (-1,-1).. controls (-1,-1.2) and (-0.5,-1.3)..(-0.5,-1.3);        
           
       \draw[-] (0.15,-1.5)--(0.15,-2);
        \draw[-](-0.15,-1.5)--(-0.15,-2);
 \draw (-0.15,-1.5).. controls (-0.14,-1.3) and (0.16,-1.3)..(0.15,-1.5);
 \draw (-0.15,-2).. controls (-0.14,-1.8) and (0.16,-1.8)..(0.15,-2);

           \draw (0.15,-1.5)..controls (0.15,-1.6) and (0.5,-1.5)..(0.5,-1.3);
            \draw (-0.15,-1.5)..controls (-0.15,-1.6) and (-0.5,-1.5)..(-0.5,-1.3);

          \node at (0,-2.2) {$\textup{td}$};
\draw[->, very thick, densely dotted, blue] (0,-0.7)--(0,2.5);
\node[very thick, blue] at (0.5,2.5) {$\R_+$};
\draw[->, very thick, densely dotted, blue] (0,-0.7)--(-2.5,-2);
\node[very thick, blue] at (-2.5,-1.5) {$\overline{M}$};
\draw[->, very thick, densely dotted, blue] (0,-0.7)--(2.5,-2);
\node[very thick, blue] at (2.5,-1.5) {$\overline{M}$};

\end{tikzpicture}

\caption{The heat blowup space $HM_{\phi}$.}
\label{fig:heat-blowup}
\end{figure}

The heat space is obtained by a sequence of blowups. 
First, one blows up $\partial M \times \partial M \times \R_+$, 
which defines a new boundary hypersurface $\textup{ff}$. Then one
blows up the lifted interior fibre diagonal, defined as above, which 
introduces a new boundary face $\textup{fd}$. Finally, one blows up
the so-called (lifted) \emph{temporal} diagonal 
$$
\textup{diag} = \{(p,q,t) \in \overline{M}\times \overline{M}\times [0,\infty) : 
p = q, t = 0\}.
$$
This final blowup introduces the boundary hypersurface $\textup{td}$.
The heat blowup space comes with a 
canonical blowdown map
\[
\beta_{t, \phi}\colon HM_{\phi} \to \overline{M}\times \overline{M}\times [0,\infty).
\]
The heat kernel, viewed as a section of the endomorphism bundle with 
appropriately rescaled $\phi$-cotangent bundles (see \cite[Definition 2.4]{MohammadBoris},
lifts to a polyhomogeneous section on $HM_{\phi}$.

\begin{thm}[{\cite[Corollary 7.2]{MohammadBoris}}] \label{thm:ShortHeatKernelPoly}
The heat kernel of $\Delta_k$ lifts to a polyhomogeneous section $\beta^*_{t,\phi} e^{-t\Delta_k}$
of the pullback bundle $\beta^*_{t,\phi} ((\Lambda^k_\phi M \otimes E) \boxtimes (\Lambda^k_\phi M \otimes E))$
on the heat space $HM_\phi$, vanishing to infinite order at $\textup{ff}$, $\textup{tf}$, $\textup{rf}$ and 
$\textup{lf}$, smooth at $\textup{fd}$, and 
of order $(-m)$ at $\textup{td}$. 
\end{thm}

Note that in contrast to the previous section, $\beta^*_{t,\phi} e^{-t\Delta_k}$ was not studied
as a half-density, but simply as a section of the endomorphism bundle. Hence the following
is a simple consequence, restricting $\beta^*_{t,\phi} e^{-t\Delta_k}$ to the lifted diagonal in 
$HM_\phi$. Note that the lifted diagonal is simply $\overline{M} \times [0,\infty)_{t}$.

\begin{Cor} \label{cor:ShortHeatKernelPoly}
The pointwise heat kernel trace $\tr (e^{-t\Delta_k})$ is a polyhomogeneous function on 
the lifted diagonal $\overline{M} \times [0,\infty)_{t} \subset HM_\phi$, 
smooth at $\partial M$ and with uniform index set $-m + 2 \N$ \textup{(}
referring to powers of $\sqrt{t}$ \textup{)} in its asymptotics as $t \to 0$.
\end{Cor}

\begin{proof}
This is an obvious consequence of Theorem \ref{thm:ShortHeatKernelPoly}. 
The only additional statement is the precise structure of the index set in the asymptotics as $t \to 0$.
This follows by the classic observation that in the interior of any smooth manifold, 
the terms $t^{-\frac{m+j}{2}}$ in the pointwise heat kernel asymptotics as $t\to 0$
vanish unless $j \in \N$ is even. 
\end{proof}

\section{Asymptotics of the renormalized heat trace on $\phi$-manifolds}

This section is devoted to deriving the asymptotic behaviour of the renormalized trace of
$e^{-t\Delta_k}$ for small times $t \in (0,1)$ from Corollary \ref{cor:ShortHeatKernelPoly}
and for large times $t\in [1,\infty)$ from the low energy behaviour
of the resolvent in Theorem \ref{thm:LongHeatKernelPoly}. Recall that, in the notation of
Theorem \ref{thm:LongHeatKernelPoly}, we now identify the pointwise trace 
$\tr(e^{-t\Delta_k}) = F' \cdot d\Vol_{\phi}$ with the 
scalar function $F'$.

\begin{Prop}\label{T}
Consider the boundary collar $[\, 0,1) \times \partial M \subset \overline{M}$ and abbreviate
$\overline{M}(\varepsilon) := \overline{M} \backslash [0,\varepsilon) \times \partial M$.
Consider the partially integrated heat trace 
$$
\mathscr{T}(t, \varepsilon):= \int_{\overline{M}(\varepsilon)} \tr (e^{-t\Delta_k}) d\Vol_{\phi}.
$$
\begin{enumerate}
\item $\mathscr{T}$ lifts to a polyhomogeneous function on the parabolic\footnote{i.e.
we treat $1/ \sqrt{t}$ as a smooth variable.} blowup space (cf. Fig. \ref{R2blowup}) 
$$
(\R_+^2)_b:= \bigl[[0,\infty)_{1/\sqrt{t}} \times 
[0,\infty)_{\varepsilon}, \{t=\infty, \varepsilon = 0\}\bigr],
$$
of order $(-b-1)$ at the left boundary face $\textup{lf}$, and order $0$ at other boundary faces.
\item $\mathscr{T}$ is polyhomogeneous on $[0,\infty)_{\sqrt{t}} \times [0,\infty)_{\varepsilon}$, 
smooth at $\varepsilon = 0$ and with the index set $-m + 2 \N$
\textup{(} referring to powers of $\sqrt{t}$ \textup{)} in its asymptotics as $t \to 0$.
\end{enumerate}
\end{Prop}

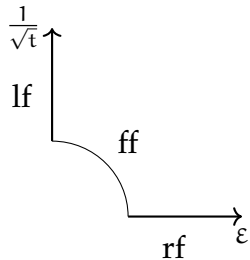
\begin{figure}[h]
\centering
\begin{tikzpicture}
\node at (1,1) {$\textup{ff}$};
\node at (-0.4,1.6) {$\textup{lf}$};
\node at (1.6,-0.4) {$\textup{rf}$};
\draw[thick,->] (1,0) -- (2.5,0) node[anchor=north] {$\varepsilon$};;
\draw[thick,->] (0,1) -- (0,2.5) node[anchor=east] {$\frac{1}{\sqrt{t}}$};;
\draw (1,0) arc (0:90:1cm);
\end{tikzpicture}
\caption{The blowup space $(\R_+^2)_b$.}
\label{R2blowup}
\end{figure}

\begin{proof} We apply Proposition \ref{auxiliary-pushforward},
where the argument remains unchanged if $[0,\infty)_{x_1}$ is replaced by 
$[0,\infty)_{1/\sqrt{t}}$ and $[0,\infty)_{x_2}$ is replaced by $\overline{M}$.
Note that 
$$
 \tr (e^{-t\Delta_k}) d\Vol_{\phi} =  x^{-(b+1)} \tr (e^{-t\Delta_k}) \nu_b,
$$
where $\nu_b$ is a $b$-density on $\overline{M}$, i.e. a smooth measure on 
$\overline{M}$, divided by the boundary defining function $x$ (which extends the $x$-coordinate of a collar neighbourhood). The factor $x^{-(b+1)}$ shifts the $\phi \textup{f}_0$ 
and $\textup{sc}$ index sets of $\tr (e^{-t\Delta_k})$
in Theorem \ref{thm:LongHeatKernelPoly} by $-(b+1)$. Hence by 
Proposition \ref{auxiliary-pushforward} $\mathscr{T}$ lifts to $(\R_+^2)_b$ with the index sets 
$$
\mathcal{E}'_{\textup{sc}} - (b+1) \geq -(b+1), \quad
\mathcal{E}'_{\textup{zf}}\geq 0, \quad
\Bigl( \mathcal{E}'_{\phi \textup{f}_0} -(b+1)\Bigr) \, \overline{\cup} \, \mathcal{E}'_{\textup{zf}}  \geq 0,
$$
at $\textup{lf}$, $\textup{rf}$ and the front face $\textup{ff}$, respectively. 
This proves the first statement. The second statement is a direct consequence of 
Corollary \ref{cor:ShortHeatKernelPoly}. 
\end{proof}

\subsection{The regularized heat trace and its asymptotics}\label{Section:RegularizedHeat}

We introduce the regularized limit $\LIM_{\varepsilon \to 0}$ of a function with an 
asymptotic expansion as $\varepsilon \to 0$ and the regularized integral (if well-defined) as 
$$
\dashint_M := \LIM_{\varepsilon \to 0} \int_{\overline{M}(\varepsilon)}.
$$
Using such notation, we can now define the regularized trace.

\begin{Def}\label{reg-trace-def} The regularized heat trace $\Tr^R (e^{-t\Delta_k})$ is the constant term 
in the asymptotic expansion of $\mathscr{T}(t, \varepsilon)$ as $\varepsilon \to 0$, denoted in the notation of Proposition \ref{T} by
$$
\Tr^R (e^{-t\Delta_k}) := \LIM_{\varepsilon \to 0} \mathscr{T}(t, \varepsilon)
\equiv \dashint_M \tr (e^{-t\Delta_k}) d\Vol_{\phi}.
$$
\end{Def}

The following observation is fairly obvious.
\begin{Lem}
\label{Lem:Cutoff}
The regularized heat trace is independent of choice of boundary defining function extending the coordinate function $x$.
\end{Lem}

\begin{Prop} \label{Lem:RegTracePoly}
The regularized heat trace satisfies \textup{(} recall $m= \dim M, b= \dim B$ \textup{)}
\begin{equation}
\begin{split}
&\Tr^R (e^{-t\Delta_k}) \sim \sum_{j = 0}^\infty a_j t^{-\frac{m}{2} + j}, \ \textup{as} \ t \to 0, \\
&\Tr^R (e^{-t\Delta_k}) \sim \sum_{j = 0}^\infty \sum_{\ell = 0}^{\ell_j} b_{j, \ell} t^{-\frac{j}{2}} \log^\ell(t), \ \textup{as} \ t \to \infty.
\end{split}
\end{equation}
for some coefficients $a_j, b_{j, \ell} \in \R$ and $\ell_j \in \N$.
\end{Prop}

\begin{proof}
The first asymptotics (as $t\to 0$) is obtained from polyhomogeneity of $\mathscr{T}$ on 
$[0,\infty)_{t} \times [0,\infty)_{\varepsilon}$ as follows: $\LIM_{\varepsilon \to 0} \mathscr{T}(t, \varepsilon)$
picks the constant term in the asymptotics of $\mathscr{T}(t, \varepsilon)$ as $\varepsilon \to 0$. By polyhomogeneity,
this term admits an asymptotic expansion as $t \to 0$
with the index set $-m + 2 \N$ \textup{(}referring to powers of $\sqrt{t}$ \textup{)} 
as asserted in Proposition \ref{T}. \medskip

For the second asymptotics (as $t \to \infty$), we use that by Proposition \ref{T}, $\mathscr{T}$ lifts to 
a polyhomogeneous function on the blowup space $(\R_+^2)_b$. Near the intersection between the front face 
$\textup{ff}$ and $\textup{lf}$, i.e. the upper corner in Figure \ref{R2blowup},
we can use projective coordinates $\tau = 1/ \sqrt{t}$
and $s = \varepsilon / \tau$. Here, $\tau$ is the defining function of $\textup{ff}$ and $s$ the defining function of
$\textup{lf}$. The lift of $\mathscr{T}$ is of order $(-b-1)$ at $\textup{lf}$ and of leading order $0$ at $\textup{ff}$, hence
we have the following expansion
$$
\mathscr{T}(t, \varepsilon) \sim \sum_{i = -(b+1)}^\infty \sum_{j = 0}^{j_i} c_{i,j} (\tau)
s^{i} \log^j(s), \ \textup{as} \ s \to 0,
$$
for some $j_i \in \N$, where each coefficient has an asymptotic expansion 
$$
c_{i,j} (\tau) \sim \sum_{\ell = 0}^\infty \sum_{p = 0}^{p_\ell} d^{i,j}_{\ell, p}
\tau^{i} \log^j(\tau), \ \textup{as} \ \tau \to 0.
$$
In particular, we conclude for the renormalized heat trace
$$
\Tr^R (e^{-t\Delta_k}) \equiv c_{0,0} (\tau) \sim 
\sum_{\ell = 0}^\infty \sum_{p = 0}^{p_\ell} d^{0,0}_{\ell, p}
\tau^{i} \log^j(\tau), \ \textup{as} \ \tau \to 0.
$$
This is precisely the claim. 
\end{proof}

\subsection{Different regularizations}\label{finite-part-sec}

We can rewrite the partially integrated heat trace as follows.
Recall the notation $\overline{M}(\varepsilon) := \overline{M} \backslash [0,\varepsilon) \times \partial M$
from Proposition \ref{T}. Then, noting that $d\Vol_{\phi} = x^{-2-b} dx \dVol_{\partial_M}(x)$, we obtain
$$
\mathscr{T}(t, \varepsilon) = \int_{\overline{M}(1)} \tr (e^{-t\Delta_k}) d\Vol_{\phi} +  
\int_{\varepsilon}^1 x^{-2-b} \left( \int_{\partial M} \tr (e^{-t\Delta_k}) \dVol_{\partial_M}(x) \right) dx.
$$
Let us write $e^{-t\Delta_k}_{\sigma}:= x^{\sigma} e^{-t\Delta_k}$ for any $\sigma \in \C$. 
Here $x$ denotes the multiplication operator, multiplying by the boundary defining function (extending the collar neighbourhood coordinate $x$)
For $\textup{Re}(\sigma)\gg 0$ we find that 
\begin{align*}
\mathscr{T}(t, \varepsilon, \sigma) &:= \int_{\overline{M}(1)} \tr (e^{-t\Delta_k}) d\Vol_{\phi} + 
\int_0^1 x^{-2-b} \left( \int_{\partial M} \tr (e^{-t\Delta_k}_{\sigma}) \dVol_{\partial_M}(x) \right) dx \\
&= \int_{\overline{M}(1)} \tr (e^{-t\Delta_k}) d\Vol_{\phi} + 
\int_0^1 x^{-2-b+\sigma} \left( \int_{\partial M} \tr (e^{-t\Delta_k}) \dVol_{\partial_M}(x) \right) dx < \infty.
\end{align*}
In fact, due to the asymptotics in Proposition \ref{T}, $\mathscr{T}(t, \varepsilon, \sigma)$ extends to a meromorphic
function in $\sigma$ on the whole of $\C$. It is a well-known fact, see for instance \cite[Proposition 2.1.7]{LeschBook},
that the renormalized trace equals the finite part of $\mathscr{T}(t, \varepsilon, \sigma)$ at $\sigma = 0$, namely
\begin{equation}
\Tr^R (e^{-t\Delta_k}) = \LIM_{\varepsilon \to 0} \mathscr{T}(t, \varepsilon) = \underset{\sigma = 0}{\FP} \ \mathscr{T}(t, \varepsilon, \sigma).
\end{equation}
This second regularisation will be convenient for later arguments.

\section{Renormalized analytic torsion on $\phi$-manifolds}

We now have everything in place to define the renormalized analytic torsion of
$(M,g_\phi)$. We begin with a standard consequence of 
Proposition \ref{Lem:RegTracePoly}.
 
\begin{Lem}
\label{Lem:zetaDef}
We define the integrals
\begin{equation}
\begin{split}
I_0^{(k)}(s):= \int_0^1 t^{s-1} \Tr^R \left( e^{-t\Delta_k} \right) \, dt, \quad \Re(s) \gg 0, \\
I_\infty^{(k)}(s):= \int_1^\infty t^{s-1}\Tr^R \left( e^{-t\Delta_k} \right) \, dt, \quad \Re(s) \ll 0.
\end{split}
\end{equation}
Then both integrals converge and admit meromorphic extensions to $\C$. 
\end{Lem}

We can now define a meromorphic function 
\begin{align}\label{zeta-formula}
\zeta(s,\Delta_k):= \frac{1}{\Gamma(s)}\left(I_0^{(k)}(s)+I_{\infty}^{(k)}(s)\right).
\end{align}
We define $\zeta_{\textup{reg}}(s,\Delta_k)$ to be the regular part of $\zeta$ near $s=0$. 

\begin{Rem}
If one can show that the $\log$-terms are absent in the two expansion of 
Proposition \ref{Lem:RegTracePoly}, then $I_0+I_\infty$ has at most a simple pole at 
$s=0$ which is cancelled by $\Gamma(s)$. This is the case we can set 
$\zeta_{\textup{reg}}(s,\Delta_k) := \zeta(s,\Delta_k)$.
\end{Rem}

\begin{Def}\label{torsion-def}
The renormalized analytic torsion of $(M,g_\phi)$ is defined by
\begin{equation}
\log T(M,E;g_\phi):= \frac{1}{2}\sum_{k=0}^{\dim M} (-1)^k\cdot k\cdot \zeta'_{\textup{reg}}(s=0,\Delta_k).
\label{eq:LogTDef}
\end{equation}
The renormalized analytic torsion norm is defined by 
$\| \cdot \|^{\textup{RS}}_{(M,E,g_\phi)} := T(M,E;g_\phi) \| \cdot \|_{L^2}$, where $ \| \cdot \|_{L^2} $ is the norm 
on $\det H^*_{(2)} (M,E)$, induced by $g_\phi$ on harmonic forms. \medskip

Let $(M,g_\phi)$ have two boundary components
$\partial M_1 \sqcup \partial M_2$, where $\partial M_1$ is the total space of a fibration and $g_\phi$ has the structure
of a $\phi$-metric in an open neighbourhood of $\partial M_1$. Let $\partial M_2$ be a regular boundary. Then, 
posing relative or absolute boundary conditions at $\partial M_2$, we may define the corresponding torsion norms
$$
\| \cdot \|^{\textup{RS}}_{(M,\partial M_2, E, g_\phi)}, \quad \| \cdot \|^{\textup{RS}}_{(M,E,g_\phi)}.
$$
\end{Def}

\section{Renormalized analytic torsion under metric perturbations}\label{Section:Variation}

We now turn to the invariance properties of the renormalized analytic torsion. 
Let $g_\phi(\gamma)$ denote a smooth family of $\phi$-metrics with 
a parameter $\gamma \in \R$, viewed as a variation of $g_\phi \equiv g_\phi(0)$. We shall also 
abbreviate $\delta =\partial_\gamma \restriction_{\gamma=0}$.
We will need to assume that the variation of the metric decays sufficiently fast near the boundary.

\begin{Assump}
The smooth family $g_\phi(\gamma), \gamma \in \R$ of $\phi$-metrics is of the form 
$$
g_\phi(\gamma) = g_\phi + h(\gamma),
$$
where the higher order terms $h(\gamma)$ vanish sufficiently fast at the boundary, 
namely $\vert h(\gamma)\vert_{g_\phi} =\mathcal{O}(x^{b+1+\alpha})$ as $x\to 0$
in a boundary collar neighborhood $\cU = (0,1)_x \times \partial M$, for some $\alpha > 0$ and $b = \dim B$. 
\label{Assumption:VarTraceClass}
\end{Assump}

We write $N$ for the degree
operator, acting as multiplication by $k$ on differential forms of degree $k$. 
We write 
$e^{-t\Delta_k}_{\sigma}:= x^{\sigma} e^{-t\Delta_k}$ for the modified heat operator . This is trace class in $L^2(\Lambda^k T^*M, g_\phi)$
for $\Re(\sigma) \gg 0$ sufficiently large.  We write $\Tr$ for the trace and, using the conventions above, 
abbreviate (in the second relation we assume 
$\Re(\sigma) \gg 0$ to ensure the trace class property)
\begin{equation}\label{TT}
\begin{split}
\Tr^R \Bigl((-1)^N Ne ^{-t\Delta}\Bigr) &=
\sum_{k=0}^m (-1)^k \cdot k\cdot \Tr^R \left(e^{-t\Delta_k}\right), \\
\Tr \Bigl((-1)^N Ne_{\sigma}^{-t\Delta}\Bigr) &=
\sum_{k=0}^m (-1)^k \cdot k\cdot \Tr \left(x^{\sigma} e^{-t\Delta_k}\right).
\end{split}
\end{equation}

\begin{Prop}\label{delta-prop} 
Let $*(\gamma)$ denote the family of Hodge star operators of $(M,g_\phi(\gamma))$. 
 We shall write $*$ for the Hodge star operator of $g_\phi \equiv g_\phi(0)$ and set 
\[
L:= *^{-1} \delta(*).
\]
Under the Assumption \ref{Assumption:VarTraceClass}, $L$ is a trace class operator in
in $L^2(\Lambda_\phi^* M, g_\phi)$ and we obtain, using the short notation 
as in \eqref{TT}
\begin{equation}
\begin{split}
\delta \Tr^R \Bigl((-1)^N Ne^{-t\Delta}\Bigr) =
- t \frac{d}{dt} \Tr \left( (-1)^N L e^{-t\Delta}\right).
\end{split}
\end{equation}
Note how the right hand side is the usual trace. 
\end{Prop}

\begin{proof}
First of all we make an easy observation that for a 
scalar function $f(\varepsilon, \gamma)$ with full asymptotic 
expansion as $\varepsilon \to 0$ and coefficients depending 
smoothly on $\gamma$, we can interchange the limits
$$
\delta \LIM_{\varepsilon \to 0} f(\varepsilon, \cdot) = 
\LIM_{\varepsilon \to 0} \, \delta f(\varepsilon, \cdot).
$$
Consequently, we find with
$f(\varepsilon, \gamma) := \int_{\overline{M}(\varepsilon)} \tr (e^{-t\Delta_\gamma}) d\Vol_{\phi}(\gamma)$
(cf. Definition \ref{reg-trace-def})
\begin{equation}\label{Tdelta}
\begin{split}
\delta \Tr^R \Bigl((-1)^N Ne^{-t\Delta}\Bigr) 
&= \Tr^R \Bigl((-1)^N N \delta e^{-t\Delta}\Bigr)
\\ &= \underset{\sigma = 0}{\FP} \ \Tr \Bigl((-1)^N N \delta e_\sigma^{-t\Delta}\Bigr),
\end{split}
\end{equation}
where we used the finite part regularization in § \ref{finite-part-sec} and
used the modified heat operator 
$e^{-t\Delta_k}_{\sigma}:= x^{\sigma} e^{-t\Delta_k}$, which is trace class 
for $\Re(\sigma) \gg 0$ sufficiently large.  
Using the semi-group property of the heat kernel we write
\[
\frac{e^{-t\Delta_{\gamma}}-e^{-t\Delta_{\gamma_0}}}{\gamma-\gamma_0} 
=\int_0^t \frac{\partial}{\partial \tau} \left(\frac{e^{-\tau \Delta_{\gamma}}
e^{-(t-\tau)\Delta_{\gamma_0}}}{\gamma-\gamma_0}\right)\, d\tau=
\int_0^t e^{-\tau \Delta_{\gamma}} \frac{-\Delta_{\gamma}+
\Delta_{\gamma_0}}{\gamma-\gamma_0}e^{-(t-\tau)\Delta_{\gamma_0}}\, d\tau.
\]
Taking the limit $\gamma\to \gamma_0$ we find
\[
\delta e^{-t\Delta}=-\int_0^t e^{-\tau \Delta} (\delta \Delta)e^{-(t-\tau)\Delta}\, d\tau.
\]
Plugging this into \eqref{Tdelta}, we find
\begin{equation}\label{Tdelta2}
\delta \Tr^R \Bigl((-1)^N Ne^{-t\Delta}\Bigr) =
- \underset{\sigma = 0}{\FP} \ \int_0^t \Tr \Bigl( x^{\sigma} (-1)^N N e^{-\tau \Delta} 
(\delta \Delta)e^{-(t-\tau)\Delta}\Bigr) d\tau.
\end{equation}
The following identity is well known, holds locally and as such on any 
smooth manifold, irrespective the Assumption \ref{Assumption:VarTraceClass}, 
cf. \cite[pp. 152]{RaySin:RTA}, namely
\begin{align}
\delta (\Delta) = L d^*d - d^* L d + d L d^* -d d^* L.
\label{eq:deltaNDelta}
\end{align}
Assumption \ref{Assumption:VarTraceClass} is chosen precisely such that $L$ 
is a trace-class operator. Indeed, $L$ acts pointwise as an endomorphism on the 
exterior algebra and its pointwise (trace) norm behaves under Assumption \ref{Assumption:VarTraceClass} 
near the boundary $\partial M$ as 
$$
\| \, L \, \| = O(x^{b+1+\alpha}), \ \textup{as} \ x \to 0.
$$
Furthermore, can easily check that the operators
$de^{-t\Delta},d^*e^{-t\Delta},\Delta e^{-t\Delta}$
are bounded in $L^2$ and hence $(\delta \Delta)e^{-(t-\tau)\Delta}$ 
is trace-class due to presence of the trace class operator $L$ in each term 
of $\delta (\Delta)$ in \eqref{eq:deltaNDelta}. Hence, taking finite part at
$\sigma = 0$ in \eqref{Tdelta2} amounts simply to evaluating $\sigma$ at zero and we conclude
\begin{align*}
\delta \Tr^R \Bigl((-1)^N Ne_{\sigma}^{-t\Delta}\Bigr) &=
- \int_0^t \Tr \Bigl( (-1)^N N e^{-\tau \Delta} 
(\delta \Delta)e^{-(t-\tau)\Delta}\Bigr) d\tau \\
&=- t \ \Tr \Bigl( (-1)^N N (\delta \Delta)e^{-t\Delta}\Bigr),
\end{align*}
where in the last equality we used cyclic invariance of the trace. 
From here on we can follow the classical argument by Ray and Singer
\cite[pp. 153]{RaySin:RTA} to conclude the statement using 
cyclic invariance of the trace.
\end{proof}

\begin{Cor}
\label{Cor:deltaFFormula}
Under the Assumption \ref{Assumption:VarTraceClass}
$$
\delta \log T(M,g_\phi) \equiv \left. \frac{d}{d\gamma} 
\right|_{\gamma = 0} \log T(M,g_\phi(\gamma)) = 
\left. \frac{d}{d\gamma} 
\right|_{\gamma = 0} \| \cdot \|^{-1}_{g_\phi(\gamma)}, 
$$
where $ \| \cdot \|_{g_\phi(\gamma)} $ is the norm 
on $\det H^*_{(2)} (M,E)$, induced by $g_\phi(\gamma)$ on $L^2$-harmonic forms.
In particular, the analytic torsion norm in Definition 
\ref{torsion-def} is independent of $\gamma$.
\end{Cor}

\begin{proof}
Proposition \ref{delta-prop} implies, by the classical argument as in \cite[pp. 153]{RaySin:RTA},
\begin{align*}
\delta \log T(M,g_\phi) &= \left. \frac{d}{ds} \right|_{s = 0} 
\frac{s}{2\Gamma(s)} \int_0^1 t^{s-1} \Tr \Bigl( (-1)^{N+1} 
L e^{-t\Delta_k} \Bigr) dt
\\ &+ \left. \frac{d}{ds} \right|_{s = 0} 
\frac{s}{2\Gamma(s)} \int_1^\infty t^{s-1} \Tr \Bigl( (-1)^{N+1} L 
e^{-t\Delta_k}\Bigr) dt,
\end{align*}
where as in Lemma \ref{Lem:zetaDef}, the first integral converges for $\Re(s) \gg 0$, 
the second for $\Re(s) \ll 0$, and both integrals admit a meromorphic extension to $\C$. 
Now $\Tr (L e^{-t\Delta})$ has asymptotics which are obtained along the lines of 
Proposition \ref{Lem:RegTracePoly},
\begin{equation}\label{L-asymptotics}
\begin{split}
&\Tr (L e^{-t\Delta}) \sim \sum_{j = 0}^\infty a'_j t^{-\frac{m}{2} + j}, \ \textup{as} \ t \to 0, \\
&\Tr(L e^{-t\Delta}) \sim \sum_{j = 0}^\infty \sum_{\ell = 0}^{\ell_j} b'_{j, \ell} t^{-\frac{j}{2}} 
\log^\ell(t), \ \textup{as} \ t \to \infty.
\end{split}
\end{equation}
An easy calculation shows 
\begin{align*}
&\left. \frac{d}{ds} \right|_{s = 0} 
\frac{s}{2\Gamma(s)} \int_0^1 t^{s-1} \Tr \Bigl( (-1)^{N+1} L e^{-t\Delta}\Bigr) dt 
= \LIM_{t \to 0}\Tr ((-1)^{N+1} L e^{-t\Delta}), 
\\ &\left. \frac{d}{ds} \right|_{s = 0} 
\frac{s}{2\Gamma(s)} \int_1^\infty t^{s-1} \Tr \Bigl( (-1)^{N+1} L e^{-t\Delta}\Bigr) dt
= \LIM_{t \to \infty}\Tr ((-1)^{N+1} L e^{-t\Delta}),
\end{align*}
where as before $\LIM$ denotes the constant term in the corresponding asymptotics. 
We infer from the first equation in \eqref{L-asymptotics} that $\LIM_{t \to 0}$ has no 
constant term\footnote{since $m= \dim M$ is odd}, while
$\LIM_{t \to \infty}$ does, namely $b'_{0,0}$. By \cite[Theorem VIII 5d)]{ReedSimon1},
the heat kernel converges pointwise to the kernel of the spectral projection $P_{\ker_{L^2}\Delta}$ 
onto the $L^2$-harmonic forms. $LP_{\ker_{L^2}\Delta}$
is trace class and we conclude
\[b'_{0,0} = \Tr (LP_{\ker_{L^2}\Delta}).\]
The claim now follows from the classical observation 
\[\Tr ((-1)^{N+1} LP_{\ker_{L^2}\Delta}) = \left. \frac{d}{d\gamma} 
\right|_{\gamma = 0} \| \cdot \|^{-1}_{g_\phi(\gamma)}.\]
\end{proof}

\begin{Rem}
In \cite{GuiShe} Guillarmou and Sher define analytic torsion for asymptotically 
conical manifolds, which corresponds to $\phi$-manifolds with trivial fibres. They do not 
discuss invariance of the renormalized torsion under metric deformations.
\end{Rem}

\section{A gluing formula for renormalized analytic torsion}

There are gluing formulas for the analytic torsion, which 
were established by Lesch \cite[Theorem 6.1]{Les:GFA} for 
manifolds with discrete spectrum. These were extended by the second author 
\cite[Theorem 2.9]{Ver} to non-compact manifolds with continuous
spectrum with either a spectral gap around zero or an acyclic vector bundle $E$. The gluing formula in 
\cite[Theorem 2.9]{Ver} applies to $\phi$-manifolds as well, and, after an extension, allows us to 
reproduce and generalize a result by Guillarmou and Sher \cite[Theorem 10 \& Lemma 42]{GuiShe}, 
which they attain without gluing formulas. Our result requires 
a restrictive assumption on the asymptotic end $\partial M$. 

\begin{Assump}
\label{Assumption:GaporAcyc}
We impose one of the following two assumptions.
\begin{enumerate}
\item either $H^k(\partial M,E)=0$ for all degrees $0\leq k \leq m-1=\dim(\partial M)$,
\item or $H^k(\partial M,E)=0$ for $1\leq k \leq m-2$ and 
$E$ is a trivial vector bundle over $\overline{M}$.
\end{enumerate}
\end{Assump}

Note that the second case of the assumption above covers the setting of 
Guillarmou and Sher \cite[Theorem 10 \& Lemma 42]{GuiShe}. Note also that
the second case of the assumption is not present in \cite[Theorem 2.9]{Ver}.\medskip

To state the gluing formula, we need some notation. Continuing in the notation of 
Section \ref{subsec-phi}, we write $N = (0,1] \times \partial M$
and $K= M \backslash ((0,1) \times \partial M)$. Let $\iota_1 \colon \partial M\to N$ and $\iota_2\colon \partial M\to K$ 
denote the natural inclusions of $\partial M$ as $\{1\} \times \partial M \subset N \cap K$. Define the 
corresponding complexes, where the lower index $(2)$ indicates $L^2$ integrability with respect to $g_\phi$
\begin{align*}
&\Omega^*_r(K,E):= \{\omega \in \Omega^*(K,E)\, \colon\, \iota_2^* \omega=0\}, \\
&\Omega^*_r(N,E):= \{\omega \in \Omega_{(2)}^*(N,E)\, \colon\, \iota_1^* \omega=0\}, \\
&\Omega^*_r(M,E):= \{(\omega_1,\omega_2) \in \Omega_{(2)}^*(N,E)\oplus \Omega^*(K,E)\, 
\colon\, \iota_1^* \omega_1=\iota_2^* \omega_2\}.
\end{align*}
These complexes fit into short exact sequences
\begin{align*}
&0\to \Omega^*_r(N,E)\xrightarrow{\alpha} \Omega_r^*(M,E)\xrightarrow{\beta} \Omega^*(K,E)\to 0, \\
&0\to \Omega_r^*(N,E)\oplus \Omega_r^*(K,E)\xrightarrow{\iota} \Omega^*_r(M,E)\xrightarrow{\rho}
\Omega^*(\partial M,E)\to 0,
\end{align*}
with $\iota$ being the inclusion and $\rho$ the restriction, 
$\alpha(\omega)=(\omega,0)$ and $\beta(\omega_1,\omega_2)=\omega_2$.
Repeating the argument of Vishik \cite[Proposition 1.1]{Vis:GRS}, the $L^2$ harmonic forms of 
$\Omega_r^*(M,E)$ and $\Omega^*(M,E)$ coincide. So we get long-exact sequences in 
(reduced $L^2-$) cohomology
\begin{align*}
&\dots \to H^p_{(2)}(N,\partial M,E)\xrightarrow{\alpha_*} H^p_{(2)}(M,E)\xrightarrow{\beta_*} 
H^p(K,E)\xrightarrow{\delta} H^{p+1}_{(2)}(N,\partial M,E)\to \dots, \\
&\dots \to H^p_{(2)}(N,\partial M,E)\oplus H^p(K,\partial M,E)\xrightarrow{\iota_*} H^p_{(2)}(M,E)\xrightarrow{\rho_*} H^p(\partial M,E)
\xrightarrow{\delta} \dots,
\end{align*}
where $\delta$ is the connecting homomorphism. These long exact sequences produce 
isomorphisms of the determinant line bundles (see for instance \cite[Proposition 1.12 \& Section 1.3]{Nikolaescu})
\begin{align*}
&\Phi\colon \det H^*_{(2)}(N,\partial M,E)\otimes \det H^*(K,E)\to \det H_{(2)}^*(M,E), \\
&\Psi\colon \det H^*_{(2)}(N,\partial M,E)\otimes \det H^*(K,\partial M,E)\otimes \det H^*(\partial M,E)\to \det H_{(2)}^*(M,E).
\end{align*}
With these maps, we can now write down the gluing formulas.

\begin{thm}
\label{thm:Gluing}
Let $M$ be an odd-dimensional connected $\phi$-manifold satisfying Assumptions \ref{assumption} 
and \ref{Assumption:GaporAcyc}. Let $(E,\nabla,h)$ be a flat hermitian vector bundle 
over $\overline{M}$. Then, assuming product structure near the cut $\{1\} \times \partial M$, 
the renormalized analytic torsion satisfies the following gluing formulas
\begin{align*}
&\norm{\Phi(\alpha \otimes \beta)}_{(M,E)}^{RS}
= 2^{-\frac{\chi(\partial M,E)}{2}}\norm{\alpha}_{(N,\partial M,E)}^{RS}\otimes \norm{\beta}_{(K,E)}^{RS}, \\
&\norm{\Psi(\alpha \otimes \beta\otimes \gamma)}_{(M,E)}^{RS}=\norm{\alpha}_{(N,\partial M,E)}^{RS} 
\otimes \norm{\beta}^{RS}_{(K,N,E)}\otimes \norm{\gamma}_{\det H^*(\partial M,E)},
\end{align*}
where we use notation of Definition \ref{torsion-def} for analytic torsion norms 
with relative and absolute boundary conditions at the 
regular boundary.
\end{thm}

\begin{proof}
This is a strengthening of \cite[Theorem 2.9]{Ver}, which was proved only under the first case of Assumption 
\ref{Assumption:GaporAcyc}. Here, we provide an extension to include the second case of Assumption 
\ref{Assumption:GaporAcyc}, which we assume from now on. The original proof was done under three assumptions: 
\begin{enumerate}
\item The first, \cite[Assumption 2.5]{Ver}, is that Lemma \ref{Lem:RegTracePoly} 
holds and that the dimensions of the $L^2$-kernels of the Laplacian acting on 
$k$-forms are finite for $0\leq k\leq m$. By \cite[Theorem 1]{Cohomology} the 
$L^2$-kernels can be identified with certain intersection cohomologies, hence are finite dimensional. 
\item The second assumption, \cite[Assumption 2.6]{Ver}, again follows from the polyhomogeneous 
descriptions given in Theorem \ref{thm:ShortHeatKernelPoly} and Theorem \ref{thm:LongHeatKernelPoly} 
of the heat kernel. 
\item The third assumption, \cite[Assumption 2.8]{Ver} is almost Assumption \ref{Assumption:GaporAcyc}, 
but we do not assume that the top and bottom cohomology of $\partial M$ vanish. This assumption gets 
used to ensure that the dimensions of certain cohomology groups do not jump, \cite[Proof of Theorem 10.4]{Ver}, 
\cite[Section  5.2.3]{Les:GFA}, and we quickly explain this next. 
\end{enumerate}
Let $\theta\in (0,\pi/2)$. In the above notation, introduce 
\[\mathscr{D}_{\theta}^k:= \left\{(\omega_1,\omega_2)\in \Omega^k_{(2)}(N,E)\oplus \Omega^k(K,E)\, \colon\, \cos\theta \iota_1^*\omega_1=\sin\theta \iota_2^* \omega_2\right\}.\]
The central ingredient in the proof of \cite[Theorem 2.9]{Ver} is that the cohomology groups 
$H^*_{\theta}(M,E)$ associated to the chain complex $\mathscr{D}_{\theta}^*$ have $\theta$-independent dimensions. 
We get a short exact sequence like above,
\[
0\to \Omega^*_r(N,E)\oplus \Omega^*_r(K,E) \xrightarrow{i_{\theta}} 
\mathscr{D}_{\theta}^* \xrightarrow{r_{\theta}} \Omega^*(\partial M,E),
\]
where $i_\theta$ is the inclusion and $r_{\theta}=\sin\theta \iota_1^*\omega_1+\cos\theta \iota^*_2\omega_2$. The long exact cohomology sequence associated to this tells us in view of the second case of Assumption \ref{Assumption:GaporAcyc} (which we assume here)
\begin{align}\label{inner-degree}
H^k_{\theta}(M,E)\cong H^k_{(2)}(N,\partial M,E)\oplus H^k(K,\partial M,E)
\end{align}
for $2\leq k\leq m-2$. For $k=0,1,m-1$ we consider the ends of the long exact sequence. We have 
(we are supressing $E$ in the notation to shorten the notation)
\begin{align*}
&0 \to H^0_{(2)}(N,\partial M)\oplus H^0(K,\partial M)\xrightarrow{\iota_*} H^0_{\theta}(M)\xrightarrow{\rho_*} H^0(\partial M)
\\ &\qquad \xrightarrow{\delta} H^1_{(2)}(N,\partial M)\oplus H^1(K,\partial M)\xrightarrow{\iota_*} H^1_{\theta}(M)\xrightarrow{\rho_*} 0, \\
&0 \to H^{m-1}_{(2)}(N,\partial M)\oplus H^{m-1}(K,\partial M)\xrightarrow{\iota_*} H^{m-1}_{\theta}(M)\xrightarrow{\rho_*} H^{m-1}(\partial M)
\\ &\qquad \xrightarrow{\delta} H^m_{(2)}(N,\partial M)\oplus H^m(K,\partial M)\xrightarrow{\iota_*} H^m_{\theta}(M)\xrightarrow{\rho_*} 0.
\end{align*}
The alternating sum of dimensions in an exact sequence of finite-dimensional vector spaces sums up to zero. 
Hence the two exact sequences imply 
\begin{equation}\label{rand-degree}
\begin{split}
\dim H_{\theta}^1(M)&=\dim H_\theta^0(M)-\dim H^0_{(2)}(N,\partial M)-\dim H^0(K,\partial M)\\
&-\dim H^0(\partial M)+\dim H^1_{(2)}(N,\partial M)+\dim H^1(K,\partial M), \\
\dim H_{\theta}^{m-1}(M)&=\dim H_\theta^m(M)-\dim H^m_{(2)}(N,\partial M)-\dim H^m(K,\partial M)\\
&+\dim H^{m-1}(\partial M)+\dim H^{m-1}_{(2)}(N,\partial M)+\dim H^{m-1}(K,\partial M).
\end{split}
\end{equation}
So, as soon as we argue that $\dim H_{\theta}^0(M,E)$ and $\dim H_{\theta}^m(M,E)$ are both $\theta$-independent, 
we get in view of \eqref{inner-degree} and \eqref{rand-degree} the $\theta$-independence of 
$\dim H_{\theta}^k(M,E)$ in all degrees $k$. For a trivial vector bundle $E$ we have a spanning set of pointwise linearly independent
parallel sections and thus $H_{\theta}^0(M,E) \cong H_{\theta}^0(M)^{\mathrm{rk}(E)}$.
For a connected $M$ (which we always assume) we have $H_{\theta}^0(M) \cong \R$ 
for all $\theta$ and hence
$$
H_{\theta}^0(M,E) \cong \R^{\mathrm{rk}(E)}, \quad \dim H_{\theta}^0(M,E) = \mathrm{rk}(E),
$$
independently of $\theta$. The argument for $k=m$ is similar. Hence $\dim H_{\theta}^k(M,E)$ is independent of $\theta$
in all degrees $k$ and the argument of \cite[Theorem 2.9]{Ver} carries over 
to our case.
\end{proof}

\begin{Rem}
We have to assume that the metric has product form near the gluing region $\{1\}\times \partial M$. 
This can always be achieved by modifying the metric away from the asymptotic end, 
and Corollary \ref{Cor:deltaFFormula} guarantees that the analytic torsion is left unchanged.
One can of course perform the gluing along some other neck dividing the manifold into a 
compact and a non-compact part. The only important thing is that everything takes place 
away from the asymptotic end $\{0\}\times \partial M$.
\end{Rem}

\section{Relating analytic torsions on $\phi$- and wedge manifolds}

This section is devoted to a proof of Theorem \ref{main3}. This is an application of our gluing result in Theorem \ref{thm:Gluing}.
We will also use the gluing formula for compact manifolds with wedge singularities due to 
Lesch \cite{Les:GFA}.  This will reproduce a result of Guillarmou and Sher \cite{She}, \cite{GuiShe} 
about conic degeneration using gluing results. Let us first introduce the setting.

\subsection*{The setting:}

We consider two compact manifolds with boundary $\overline{M}$ and $\overline{\Omega}$ with a common boundary, $\partial \overline{M}=\partial M=\partial \overline{\Omega}$. 
This shared boundary $\partial M = B \times F$ is product of two compact 
manifolds $B$ and $F$.  We write $\cU_\phi = (0,1]_x \times \partial M \subset M$ and 
$\cU_\omega = (0,1]_r \times \partial M \subset \Omega$ for the respective boundary collars. 
We equip the open interiors $M \subset \overline{M}$ and $\Omega \subset \overline{\Omega}$
with the Riemannian $\phi$-metric $g_\phi$ and the Riemannian wedge metric $g_\omega$, respectively, namely 
\begin{align*}
&g_\phi \restriction \cU_\phi = \frac{dx^2}{x^4}+\frac{g_{B}}{x^2}+g_{F},\\
&g_\omega \restriction \cU_\omega = dr^2 + r^2 g_B + g_F,
\end{align*}
where $g_B$ and $g_F$ are Riemannian metrics on $B$ and $F$, respectively. Note that
under a change of variables $x= 1/r$, $g_\phi$ attains the same form as $g_\omega$, namely
\begin{equation}\label{xr}
g_\phi \restriction \Bigl( \cU_\phi = [1,\infty)_r \times \partial M \Bigr) = dr^2 + r^2 g_B + g_F.
\end{equation}
We consider compactly supported perturbations $h_\phi$ and $h_\omega$ of the metrics 
$g_\phi$ and $g_\omega$, respectively, such that 
\begin{align*}
&(g_\phi + h_\phi) = dx^2 + g_B + g_F, \ \textup{near the
hypersurface}  \ \Sigma = \{1\} \times \partial M \subset M,
\\  &(g_\omega + h_\omega) = dr^2 + g_B + g_F, \ \textup{near the
hypersurface} \ \Sigma = \{1\} \times \partial M \subset \Omega.
\end{align*} 
These two manifolds are illustrated in Figure \ref{gluing-before}.

\begin{figure}[h]
\centering
\begin{tikzpicture}
\node at (0,0) {\includegraphics[scale=0.5]{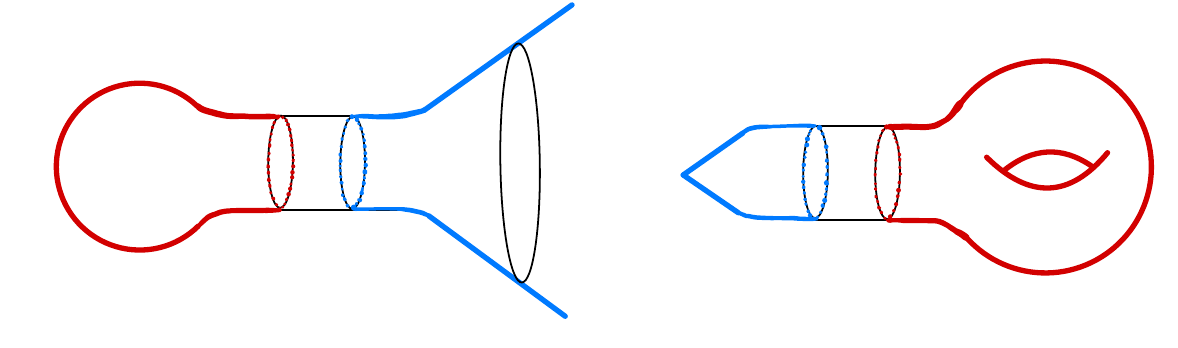}};
\node at (-4.5,1.2) {$M$};
\node at (4.6,1.3) {$\Omega$};
\node at (-2.3,1) {$\Sigma$};
\node at (2.2,0.9) {$\Sigma$};
\node at (-1.4,-1) {$\cU_\phi$};
\node at (1.5,-0.9) {$\cU_\omega$};
\end{tikzpicture}
\caption{Manifolds $(M,g_\phi + h_\phi)$ and $(\Omega, g_\omega+h_\omega)$.}
\label{gluing-before}
\end{figure}
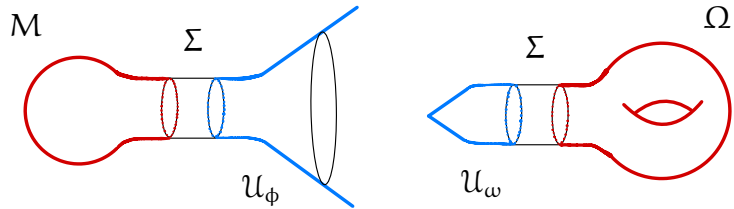

This allows us to define a new compact Riemannian manifold 
\begin{align*}
(K,g_K) := 
\Bigl(M \backslash \cU_\phi, g_\phi + h_\phi\Bigr) &\cup_{\Sigma} 
\Bigl(\Omega \backslash \cU_\omega, g_\omega + h_\omega\Bigr), \ \textup{i.e.} \\
&g_K \restriction (M \backslash \cU_\phi) := g_\phi + h_\phi, \\
&g_K \restriction (\Omega \backslash \cU_\omega) := g_\omega + h_\omega, 
\end{align*}
where $g_K$ is a smooth Riemannian metric precisely due to the product structure assumption
on the perturbations $h_\phi$ and $h_\omega$. Finally, we may also glue 
$(\cU_\phi, g_\phi)$ and $(\cU_\omega, g_\omega)$ to a model wedge manifold
\begin{align*}
(\cU,g_{\cU}) &:= 
\Bigl(\cU_\phi, g_\phi +h_\phi\Bigr) \cup_{\Sigma} 
\Bigl(\cU_\omega, g_\omega+h_\omega\Bigr), \ \textup{i.e.} \\
&g_{\cU} \restriction (\cU_\phi, g_\phi) := g_\phi + h_\phi, \\
&g_{\cU} \restriction (\cU_\omega, g_\omega) := g_\omega + h_\omega.
\end{align*}
These new two manifolds are illustrated in Figure \ref{gluing-after} and the red 
versus blue colouring corresponds to the colouring in the previous 
Figure \ref{gluing-before}.

\begin{figure}[h]
\centering
\begin{tikzpicture}
\node at (0,0) {\includegraphics[scale=0.5]{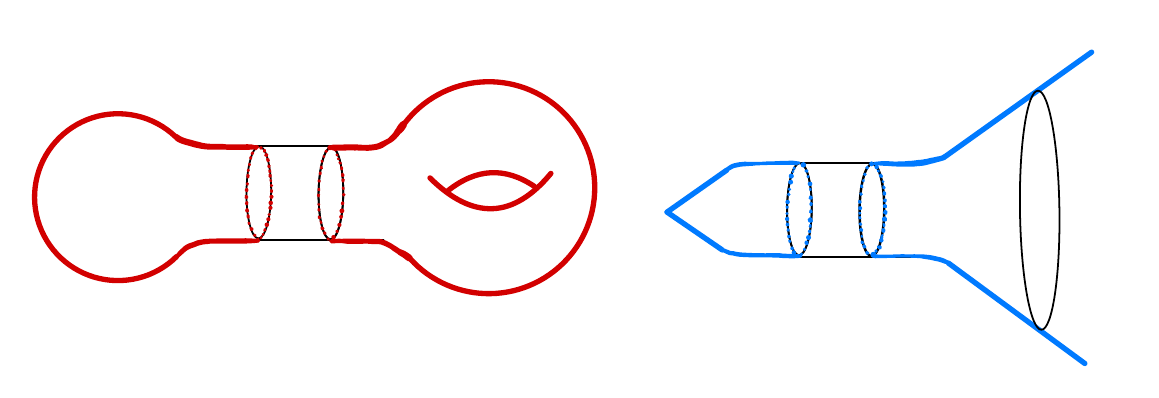}};
\node at (-4.5,1.2) {$K$};
\node at (4.8,1.3) {$\cU$};
\node at (-2.3,1) {$\Sigma$};
\node at (2.2,0.9) {$\Sigma$};
\node at (-0.7,-1.3) {$\Omega \backslash \cU_\omega$};
\node at (-3.7,-1.3) {$M \backslash \cU_\phi$};
\node at (1.5,-0.9) {$\cU_\omega$};
\node at (3.4,-1.3) {$\cU_\phi$};
\end{tikzpicture}
\caption{Manifolds $(K,g_K)$ and $(\cU, g_{\cU})$.}
\label{gluing-after}
\end{figure}

Finally, we fix flat Hermitian vector bundles over $\overline{M}$ and $\overline{\Omega}$,
which can be identified over open neighborhoods of $\Sigma$ in $M$ and $\Omega$. 
They induce flat Hermitian vector bundles over $\cU$ and $K$. And by a slight abuse of
notation we denote all these vector bundles by $E$.

\subsection*{The gluing formulas:}

We impose Assumptions \ref{assumption} and \ref{Assumption:GaporAcyc} on $(M,g_\phi)$.
Then a combination of Theorem \ref{thm:Gluing}, Lesch \cite{Les:GFA} and \cite{Vis:GRS} 
gives plethora of canonical isomorphisms of determinant line bundles 
(the $L^2$ cohomologies are defined with respect to the metrics as defined above)
\begin{align*}
&\Phi_M\colon\, \det H^*(M \backslash \cU_\phi,E)\otimes 
H^*_{(2)}(\cU_{\phi},\Sigma, E)\to \det H^*_{(2)}(M,E), \\
&\Phi_\Omega \colon\, \det H^*(\Omega \backslash \cU_\omega,\Sigma,E)\otimes 
H^*_{(2)}(\cU_{\omega},E)\to \det H^*_{(2)}(\Omega,E), \\
&\Phi_K \colon\, \det H^*(M \backslash \cU_\phi,E)\otimes 
H^*(\Omega \backslash \cU_\omega,\Sigma,E)\to \det H^*(K,E), \\
&\Phi_{\cU}\colon\, \det H^*_{(2)}(\cU_\omega,E)\otimes 
H^*_{(2)}(\cU_{\phi},\Sigma, E)\to \det H^*_{(2)}(\cU,E).
\end{align*}
These maps define isometries of the (renormalized) analytic torsion norms
\begin{equation}\label{gluing-formulae}
\begin{split}
&2^{-\frac{\chi(\Sigma,E)}{2}}
\norm{u}_{(M \backslash \cU_\phi,E, g_\phi+h_\phi)}^{RS}
\cdot \norm{v}^{RS}_{(\cU_{\phi},\Sigma, E, g_\phi+h_\phi)}
=\norm{\Phi_M(u\otimes v)}^{RS}_{(M,E,g_\phi+h_\phi)}, \\
&2^{-\frac{\chi(\Sigma,E)}{2}}
\norm{u}_{(\Omega \backslash \cU_\omega,\Sigma,E,g_\omega+h_\omega)}^{RS}
\cdot \norm{v}^{RS}_{(\cU_{\omega},E,g_\omega+h_\omega)}
=\norm{\Phi_\Omega(u\otimes v)}^{RS}_{(\Omega,E,g_\omega+h_\omega)}, \\
&2^{-\frac{\chi(\Sigma,E)}{2}}
\norm{u}_{(M \backslash \cU_\phi,E,g_\phi+h_\phi)}^{RS}
\cdot \norm{v}^{RS}_{(\Omega \backslash \cU_\omega,\Sigma,E,g_\omega+h_\omega)}
=\norm{\Phi_K(u\otimes v)}^{RS}_{(K,E,g_K)}, \\
&2^{-\frac{\chi(\Sigma,E)}{2}}
\norm{u}_{(\cU_\omega,E,g_\omega +h_\omega)}^{RS}
\cdot \norm{v}^{RS}_{(\cU_{\phi},\Sigma, E, g_\phi +h_\phi)}
=\norm{\Phi_{\cU}(u\otimes v)}^{RS}_{(\cU,E,g_{\cU})}.
\end{split}
\end{equation}

Note that the analytic torsion norm of $(\Omega, g_\omega+h_\omega)$
is defined and studied in \cite{MazVer}. The renormalized analytic torsion norm of
$(\cU,g_{\cU})$ can be defined by combining the analysis in \cite{MazVer} for the incomplete
end and the analysis behind Definition \ref{torsion-def} for the complete $\phi$-end.

\subsection*{Proof of Theorem \ref{main3}:}

The next result is an obvious rephrasing of Theorem \ref{main3}
in terms of renormalized analytic torsion norms.

\begin{thm}\label{main3-rephrazed}
We impose Assumptions \ref{assumption} and \ref{Assumption:GaporAcyc} on $(M,g_\phi)$.
Assume that $\dim F$ is even. Then there exists a canonical isomorphism of determinant lines
$$
\Phi: \det H^*_{(2)}(M,E) \otimes \det H^*_{(2)}(\Omega,E) \to \det H^*(K,E),
$$
which is an isometry of the (renormalized) analytic torsions
\begin{equation}\label{iso}
\norm{u}_{(M,E,g_\phi)}^{RS}
\cdot \norm{v}^{RS}_{(\Omega,E,g_\omega)}
=\norm{\Phi(u\otimes v)}^{RS}_{(K,E,g_K)}.
\end{equation}
Here $g_K$ can be replaced by any smooth Riemannian metric on $K$
in view of the classical Cheeger-M\"uller theorem.
\end{thm}

\begin{proof} Up to reordering of individual determinant lines, we define
\begin{align*}
&\Phi':= (\Phi_K \otimes \Phi_{\cU}) \circ (\Phi_M^{-1} \otimes \Phi_\Omega^{-1}): 
\\ &\det H^*_{(2)}(M,E) \otimes \det H^*_{(2)}(\Omega,E) \to \det H^*(K,E) \otimes \det H^*_{(2)}(\cU,E).
\end{align*}
Writing $\Phi'(u \otimes v) = \Phi'(u \otimes v)_K \otimes \Phi'(u \otimes v)_{\cU}$, 
we find using \eqref{gluing-formulae}
$$
 \norm{u}^{RS}_{(M,E,g_\phi+h_\phi)} \cdot \norm{v}^{RS}_{(\Omega,E,g_\omega+h_\omega)} = 
 \norm{\Phi'(u\otimes v)_K}^{RS}_{(K,E,g_K)} \cdot  
 \norm{\Phi'(u\otimes v)_{\cU}}^{RS}_{(\cU,E,g_{\cU})}. 
$$
Note that due to \eqref{xr}, the metric $g_{\cU}$ on $\cU = (0,\infty)_r \times \partial M$ is a
compactly supported perturbation of the exact edge metric 
$$
\overline{g}_{\cU} = dr^2 +r^2 g_B + g_F.
$$
In view of Corollary \ref{Cor:deltaFFormula} and the variational formula for analytic torsion of wedge manifolds 
in \cite{MazVer}, which are just classical well-known arguments in case of compactly supported metric 
perturbations $h_\phi$ and $h_\omega$, we find
\begin{equation}\label{intermediate-gluing}
\norm{u}^{RS}_{(M,E,g_\phi)} \cdot \norm{v}^{RS}_{(\Omega,E,g_\omega)} = 
 \norm{\Phi'(u\otimes v)_K}^{RS}_{(K,E,g_K)} \cdot  
 \norm{\Phi'(u\otimes v)_{\cU}}^{RS}_{(\cU,E,\overline{g}_{\cU})}.
\end{equation}
It will be convenient to work with a different isomorphism of determinant lines,
which employ the dual vector space $\det H^*_{(2)}(\cU,E)^{-1}$, namely
\begin{align*}
&\Phi'': \det H^*_{(2)}(M,E) \otimes \det H^*_{(2)}(\Omega,E) \otimes \det H^*_{(2)}(\cU,E)^{-1} \to \det H^*(K,E), \\
&\Phi'' (u \otimes v \otimes w^{-1}) := \Phi' (u \otimes v) \otimes w^{-1}.
\end{align*}
We compute, using \eqref{intermediate-gluing} in the last step
\begin{equation}\label{intermediate-gluing2}
\begin{split}
& \norm{\Phi'' (u \otimes v \otimes w^{-1})}^{RS}_{(K,E,g_K)} \equiv
 \norm{\Phi' (u \otimes v) \otimes w^{-1}}_{(K,E,g_K)} \\ & \qquad =
  \norm{\Phi'(u\otimes v)_K}^{RS}_{(K,E,g_K)} \cdot  
 \norm{\Phi'(u\otimes v)_{\cU}}^{RS}_{(\cU,E,\overline{g}_{\cU})}
 \cdot \Bigl( \norm{w}^{RS}_{(\cU,E,\overline{g}_{\cU})} \Bigr)^{-1} \\ &\qquad = 
 \norm{u}^{RS}_{(M,E,g_\phi)} \cdot \norm{v}^{RS}_{(\Omega,E,g_\omega)} 
 \cdot  \Bigl( \norm{w}^{RS}_{(\cU,E,\overline{g}_{\cU})} \Bigr)^{-1}.
 \end{split}
\end{equation}
This is almost the desired claim of \eqref{iso}, up to
the norm of $w$. Harmonic forms on $(\cU,\overline{g}_{\cU})$ are given by harmonic forms on $F$ 
times $r^\beta$ with $\beta \in \R$ determined by the spectrum of $B$. Such expressions 
are never in $L^2(\cU, \overline{g}_{\cU})$, unless they are zero. Hence 
\begin{equation}
H^*_{(2)}(\cU,E) = 0, \quad \det H^*_{(2)}(\cU,E) = \C,
\end{equation}
and thus $w \in \det H^*_{(2)}(\cU,E)$ is actually a complex number. 
By Proposition \ref{Prop:ExactConeTorsion}, the renormalized scalar analytic torsion of $(\cU,g_{\cU})$
equals $1$ and hence for all $w \in \det H^*_{(2)}(\cU,E) \cong \C$ the renormalized analytic torsion norm
$\norm{w}^{RS}_{(\cU,E,\overline{g}_{\cU})}$ is simply the absolute value of $w$. We conclude
from \eqref{intermediate-gluing2}
\begin{equation}\label{intermediate-gluing3}
\norm{\Phi'' (u \otimes v \otimes w^{-1})}^{RS}_{(K,E,g_K)} 
 = \norm{u}^{RS}_{(M,E,g_\phi)} \cdot \norm{v}^{RS}_{(\Omega,E,g_\omega)} 
 \cdot  \Bigl( | \, w \, | \Bigr)^{-1}.
\end{equation}
Setting $\Phi (u \otimes v):= \Phi'' (u \otimes v \otimes w^{-1}) \cdot w$, yields the claim.
\end{proof}

\appendix

\section{Microlocal preliminaries}\label{appendix-microlocal}

We recall some basic concepts of polyhomogenous asymptotic expansions and blow-ups. 
We refer the reader to \cite{Mel:TAP}, \cite[Chapter 1, Chapter 5]{Mel:Analysis} and \cite{Grieser} 
for much more detailed introductions.

\subsection*{Manifolds with corners}

An $n$-dimensional manifold with corners $X$ is a second countable Hausdorff space, locally modelled on $[0,\infty)^k \times \R^{n-k}$ for varying $k\in \{0,1,\dots,n\}$. We denote its open interior by $X^o$ and the set of boundary hypersurfaces by $\mathcal{M}_1(X)$. 
Each boundary hypersurface $H \in \mathcal{M}_1(X)$ is itself an
$(n-1)$-dimensional manifold with corners. We assume that the boundary hypersurfaces are embedded. A boundary defining function $\rho_H$ for a hypersurface $H$ is a function $\rho_H\colon X\to [0,\infty)$ such that $\rho_H^{-1}(0)=H$, $\rho_H$ is smooth up to the boundary\footnote{Smoothness can be taken to mean boundedness of all derivatives in the interior.}, and $d\rho_H\neq 0$ on $H$. 

\subsection*{Blow-ups}
Assume $X$ is a manifold with corners and $Y\subset X$ is a  $p-$submanifold. The blow-up $[X;Y]$ is defined as the space obtained by gluing together $X\setminus Y$ and the inward spherical normal bundle $S^+NY$ of $Y\subset X$. We will refer to $S^+NY$ as the front face of the blow-up. We equip the blow-up with the natural topology and the unique minimal differential structure with respect to which smooth functions with compact support in the open interior $(X\setminus Y)^o$ and polar coordinates around $Y$ in $X$ are smooth. The blow-up is equipped with the blow-down map 
\[\beta\colon [X;Y]\to X\]
which is the identity on $X\setminus Y$ and the bundle projection on $S^+NY$. We use the blow-down map to lift $p-$submanifolds $Z\subset X$ to $[X;Y]$ as follows. 
\begin{itemize}
\item If $Z\subseteq Y$, then $\beta^*(Z):= \beta^{-1}(Z)$.
\item If $Z\nsubseteq Y$, then $\beta^*(Z):= \overline{\beta^{-1}(Z\setminus Y)}$, where the closure is in $[X;Y]$.
\end{itemize}

\subsection*{Polyhomogeneous expansions}
Let $X$ be a manifold with corners. Recall that a b-vector field on $X$ is a 
smooth vector field which is tangential to all boundary faces of $X$. 
We denote by $\mathcal{V}_{b}(X)$ the space of b-vector fields on $X$.
The notion of polyhomogeneous functions uses the notion of b-vector fields and 
index sets that are defined as follows. A discrete subset $\mathcal{E}\subset \C\times \N_0$ 
is called an index set if
\begin{enumerate}
\item the set $\{\textup{Re}(\alpha)\, \vert \, (\alpha,k)\in \mathcal{E}\}$ accumulates only at $+\infty$, 
\item for each $\alpha\in \C$ there is $k_\alpha\in \N_0$ such that $(\alpha,k)\in \mathcal{E}\implies k\leq k_{\alpha}$,
\item if $(\alpha,k)\in \mathcal{E}$, then $(\alpha+\ell,k')\in \mathcal{E}$ for all $\ell\in \N_0$ and all $k'$ with $0\leq k'\leq k$.
\end{enumerate}
The extended union of two index sets is 
\begin{align}\label{extended-union}
\mathcal{E} \bcup \mathcal{F}:= \mathcal{E}\cup \mathcal{F}\cup \{(\alpha,k+\ell+1)\, \vert\, (\alpha,k)\in \mathcal{E} \, \&\, (\alpha,\ell)\in \mathcal{F}\}.
\end{align}
We will rarely specify index sets explicitly. Instead we give lower bounds, where the notation is
\begin{align*}
&\mathcal{E}> c \iff \forall (\alpha,k)\in \mathcal{E}: \, \textup{Re}(\alpha)>c, \\
&\mathcal{E}\geq c \iff \forall (\alpha,k)\in \mathcal{E}: \,  \textup{Re}(\alpha)\geq c \textup{ and } 
k=0 \textup{ if } \textup{Re}(\alpha)=c.
\end{align*}
An index family is an assignment $H\mapsto \mathcal{E}_H$ of an index set to each boundary hypersurface $H\in \mathcal{M}_1(X)$. A function $f$ is called polyhomogeneous on $X$ with index family $\mathcal{E}$ if it is smooth on $X^o$ and near each hypersurface $H\in \mathcal{M}_1(X)$ there is an asymptotic expansion
\begin{equation}
f\sim \sum_{(\alpha,k)\in \mathcal{E}} a_{\alpha,k} \rho_H^{\alpha} \log(\rho_H)^k, \ \textup{as} \ \rho_H\to 0.
\label{eq:Polyhomog}
\end{equation}
Here $\rho_H$ is any boundary defining function of $H$, and $a_{\alpha,k}$ is polyhomogeneous on $H$. The index family of the coefficients $a_{\alpha,k}$ is $\mathcal{E}^H$, defined as $H'\mapsto \mathcal{E}_{H'}$ for any $H'\in \mathcal{M}_1(X)$ with $H\cap H'\in \mathcal{M}_1(X)$. We stress that $\mathcal{E}\geq 0$ means the leading term in the polyhomogenous expansion \eqref{eq:Polyhomog} is a constant term. We assume that the asymptotic expansion 
\eqref{eq:Polyhomog} is preserved under iterated application of b-vector fields.

\section{An auxilliary result on some integrals}\label{appendix-integral}

\begin{Prop}\label{auxiliary-pushforward} Consider the blowup space
$$
X_b:= \bigl[[0,\infty)_{x_1} \times [0,\infty)_{x_2}; \{x_1=0\} \times \{x_2=0\}],
$$ 
with the blowdown map $\beta: X_b \to [0,\infty)_{x_1} \times [0,\infty)_{x_2}$. We denote 
\begin{enumerate}
\item the front face corresponding to the blowup by $(11)$, 
\item the lift of the face $\{x_1=0\} \times [0,\infty)_{x_2}$ by $(10)$, 
\item the lift of the face $[0,\infty)_{x_1} \times \{x_2=0\}$ by $(01)$. 
\end{enumerate}
Consider $u: [0,\infty)_{x_1} \times [0,\infty)_{x_2} \to \R$,
which lifts to a polyhomogeneous function on $X_b$ with compact support and index sets
$\mathcal{E}_{10}, \mathcal{E}_{11}, \mathcal{E}_{01}$ at the faces $(10), (11), (01)$, respectively.
Then
\begin{equation}\label{u-integral}
\int^\infty_{x_3} u(x_1,x_2) \frac{dx_2}{x_2} 
\end{equation}
lifts to a polyhomogeneous function on $X_b$ with 
$\mathcal{E}_{10}, \mathcal{E}_{11} \overline{\cup} \mathcal{E}_{10}, \mathcal{E}_{01}$ at the faces $(10), (11), (01)$, respectively.
\end{Prop}

\begin{proof} We employ the same formalism that is used in the 
composition of $b$-operators\footnote{We thank Daniel Grieser for pointing 
out this trick to us.}. We define 
\begin{align*}
&X^2(L):=[0,\infty)_{x_1} \times [0,\infty)_{x_2},\\
&X^2(R):=[0,\infty)_{x_2} \times [0,\infty)_{x_3},\\
& X^2(C):=[0,\infty)_{x_1} \times [0,\infty)_{x_3},
\end{align*}
and their blowups
\begin{align*}
&X^2_b(L):= [X^2(L); \{x_1=0\}\times \{x_2=0\}],
\\ &X^2_b(R):= [X^2(R); \{x_2=0\}\times \{x_3=0\}],
\\ &X^2_b(C):= [X^2(C); \{x_1=0\} \times \{x_3=0\}].
\end{align*}
We denote all the corresponding blowdown 
maps by $\beta$ by a small abuse of notation. We construct the so-called triple
space $X^3_b$, which is a blowup of $X^3 := [0,\infty)_{x_1} \times [0,\infty)_{x_2} \times [0,\infty)_{x_3}$,
obtained by blowing up the highest codimension corner $\{x_1=x_2=x_3=0\}$ and then the
lifts of the axes $\{x_1=x_2=0\}$, $\{x_1=x_3=0\}$ and $\{x_2=x_3=0\}$. The blowup is chosen such that
the natural projections of $X^3$ to $X^2(L), X^2(R)$ and $X^2(C)$ lift to $b$-fibrations, an 
important class of maps introduced in \cite{Mel:push}. 
\begin{align*}
\pi_L: X^3_b \to X^2_b(L), \quad 
\pi_R: X^3_b \to X^2_b(R), \quad 
\pi_C: X^3_b \to X^2_b(C).
\end{align*}

\begin{figure}[h]
\includegraphics[scale=0.6]{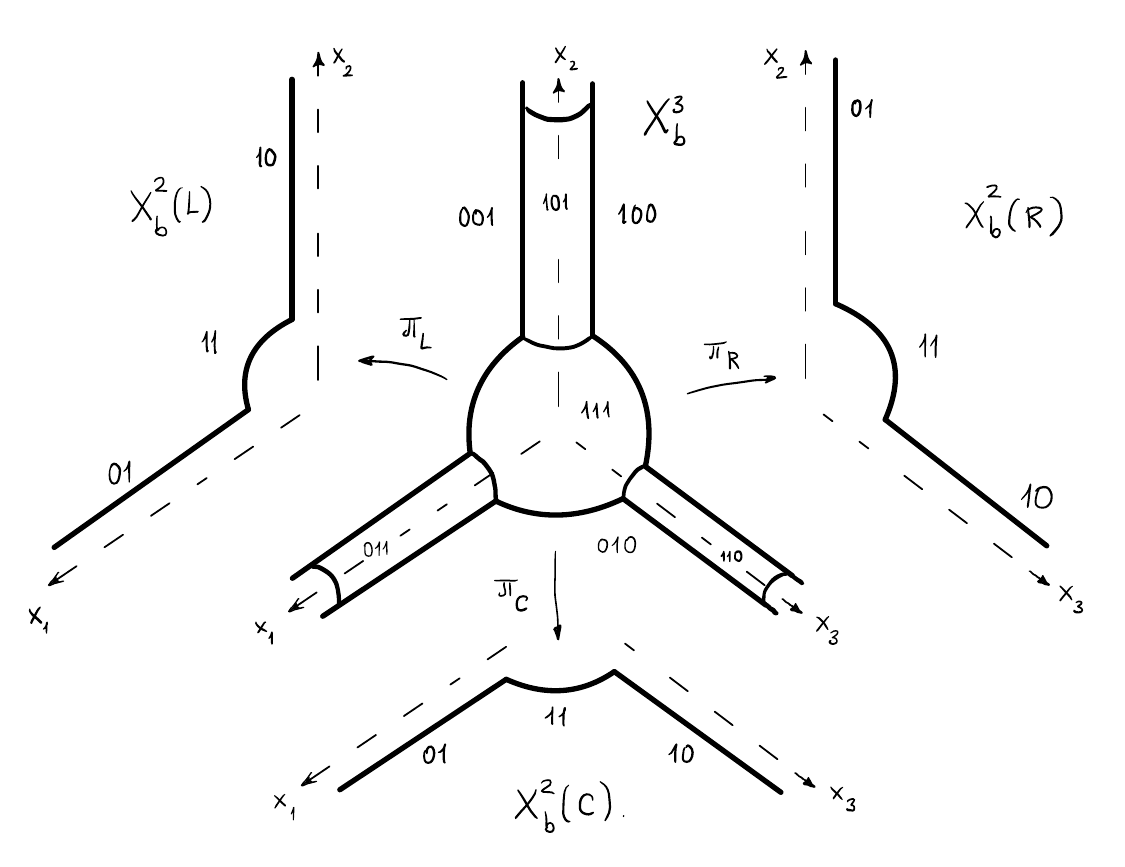}
\caption{$b$-triple space}
\label{triple-figure}
\end{figure}

These spaces and maps are illustrated in Figure \ref{triple-figure}, and a consequence
of the celebrated Melrose's pushforward theorem in \cite{Mel:push} states that if
\begin{enumerate}
\item $u$ lifts to a polyhomogeneous function on $X^2_b(L)$ with compact support and index sets
$\mathcal{E}_{10}, \mathcal{E}_{11}, \mathcal{E}_{01}$ at the faces $(10), (11), (01)$, respectively; 
\item $w$ lifts to a polyhomogeneous function on $X^2_b(R)$ with compact support and index sets
$\mathcal{E}'_{10}, \mathcal{E}'_{11}, \mathcal{E}'_{01}$ at the faces $(10), (11), (01)$, respectively, 
\end{enumerate}
then the pushfoward
\begin{align*}
&(\pi_C)_* \left( \beta^*u \cdot \beta^*w \cdot \beta^* \Bigl( \frac{dx_1}{x_1}  \frac{dx_2}{x_2}\frac{dx_3}{x_3} 
\Bigr)\right) \\ &= \beta^* \left( \int_0^\infty u(x_1,x_2) w(x_2,x_3) \frac{dx_2}{x_2} 
\right) \frac{dx_1}{x_1}  \frac{dx_3}{x_3} =: \beta^* f \cdot \frac{dx_1}{x_1} \frac{dx_3}{x_3}
\end{align*}
is a polyhomogeneous density on $X^2_b(C)$, where the index sets of $\beta^* f$ are given 
in terms of \emph{extended unions} $\overline{\cup}$ defined in \eqref{extended-union}
\begin{equation}\label{pushforward-general}
\begin{split}
&\mathcal{F}_{10} = (\mathcal{E}_{11} + \mathcal{E}'_{10}) \overline{\cup} \mathcal{E}_{10}, \\
&\mathcal{F}_{11} = (\mathcal{E}_{10} + \mathcal{E}'_{01}) \overline{\cup} (\mathcal{E}_{11} + \mathcal{E}'_{11}), \\
&\mathcal{F}_{01} = (\mathcal{E}_{01} + \mathcal{E}'_{11}) \overline{\cup} \mathcal{E}'_{01}.
\end{split}
\end{equation}
The function $f$ is given by \eqref{u-integral}, if we set $w \equiv 1$ for $x_3 \leq x_2$ and $0$ otherwise.
The lift $\beta^*w$ is polyhomogeneous (the discontinuity at $x_2=x_3$ is irrelevant for the argument)
with index sets $\mathcal{E}'_{10} = \varnothing, \mathcal{E}'_{11} = \mathbb{N}_0 \times \mathbb{N}_0, 
\mathcal{E}'_{01}\geq \mathbb{N}_0 \times \mathbb{N}_0$. 
Here $\mathcal{E}'_{10} = \varnothing$ means that $\beta^*w$ is vanishing to infinite order at $(10)$, in 
fact it is even identically zero there. Thus \eqref{pushforward-general} implies the statement for these particular
index sets.
\end{proof}

\section{Renormalized analytic torsion of a wedge}\label{appendix-wedge}
Let $(B,g_B)$, $(F,g_F)$ be two compact Riemannian manifolds. The \textit{Riemannian cone over $B$} is $\mathscr{C}(B)\coloneqq (0,\infty)\times B$ equipped with the warped product metric $g_{C(B)}\coloneqq dr^2+r^2 g_B$. The \textit{Model wedge} is  $\cU = \mathscr{C}(B) \times F$, with the product metric $g_{\cU} = g_{\mathscr{C}(B)} + g_F$. We should point out that
on wedges, $F$ is usually reserved for the cone link and $B$ for the fibration base (in this appendix the fibration is trivial). Their roles are reversed here to match  the $\phi$-setting as in \eqref{reversed} and the rest of the paper.

\begin{Prop}
\label{Prop:ExactConeTorsion}
Assume that $\dim F$ is even. Then the renormalized scalar analytic torsion of 
the model wedge $(\cU,g_{\cU})$
equals $1$ for any choice of $(B,g_B)$ and $(F,g_F)$.
\end{Prop}

\begin{proof}
We will as usual suppress the vector bundle $E\to \cU$ in the notation, and 
just write $\Omega^*(\cU)$ when we mean $\Omega^*(\cU,E)$.
Consider $\Delta$, the relative or absolute self-adjoint extension 
of the Hodge Laplacian on $(\cU,g_{\cU})$. See
\cite{Hilbert} for the discussion of relative and absolute boundary conditions,
see also \cite[(2.5)]{MV}. \medskip

\noindent\textbf{Case 1 - $F=\{pt\}$:} \medskip

\noindent
We will actually prove a bit more, namely that the associated zeta-functions 
$\zeta(s,\Delta_{\mathscr{C}(B)})$ vanish identically in each degree.
Let $\lambda>0$ and introduce the rescalings 
\begin{align*}
&f_{\lambda}\colon \mathscr{C}(B) \to \mathscr{C}(B),\quad 
f_{\lambda}(r,y)=(\lambda r, y), \\
&F_{\lambda}\colon (0,\infty)\times \mathscr{C}(B) \to (0,\infty)\times \mathscr{C}(B),\quad 
F_{\lambda}(t,r,y)=(\lambda^2 t, \lambda r, y).
\end{align*}
$f_{\lambda}$ is a smooth map, so $f_{\lambda}^* d= d f^*_{\lambda}$. If $\tilde{g}=\lambda^2 g$, then $\tilde{*}_{\vert \Omega^p}=\lambda^{m-2p}*_{\vert \Omega^p}$. So $f^*d^*=\lambda^{-2} d^* f^*$. So 
\[
f^*_{\lambda} \Delta=\lambda^{-2} \Delta f^*_\lambda.
\]
From here we conclude that the rescalings preserve the relative and absolute domains 
of $\Delta_{\mathscr{C}(B)}$. Now, one easily checks that $\omega$ solves the 
heat equation with initial data $\omega_0$ if and only if $F^*_{\lambda}\omega$ 
solves the heat equation with initial data $f_{\lambda}^* \omega_0$. Hence,
with $\dim B = b$,
\begin{align*}
\omega(\lambda^2 t, \lambda r, y)
&= \int_0^\infty \int_B e^{-t\Delta}(t, r, y, \tilde{r}, \tilde{y}) \ \omega_0(\lambda \tilde{r}, \tilde{y})\, \tilde{r}^b d\tilde{r} \, \dVol_{g_B}(\tilde{y}) \\
&= \int_0^\infty \int_B \lambda^{-1-b} e^{-t\Delta}(t, r, y, \tilde{r}/\lambda, \tilde{y}) \ \omega_0(\tilde{r}, \tilde{y})\, \tilde{r}^b d\tilde{r} \,\dVol_{g_B}(\tilde{y}).
\end{align*}

On the other hand, $\omega(\lambda^2 t, \lambda r, y)$ can be obtained as
\begin{align*}
\omega(\lambda^2 t, \lambda r, y)
= \int_0^\infty \int_B e^{-t\Delta}(\lambda^2 t, \lambda r, y, \tilde{r}, \tilde{y}) \ \omega_0(\tilde{r}, \tilde{y})\, \tilde{r}^b d\tilde{r} \, \dVol_{g_B}(\tilde{y}).
\end{align*}
Comparing these two expressions, we obtain
\begin{align}\label{rescaling}
e^{-t\Delta}(\lambda^2 t, \lambda r, y, \lambda \tilde{r}, \tilde{y}) = \lambda^{-1-b} e^{-t\Delta}(t, r, y, \tilde{r}, \tilde{y}).
\end{align}
Note that here $(1+b)$ is the total dimension of $\mathscr{C}(B)$.
Similar results were discussed in \cite[Section 2]{Che:SGS}, for instance.
We now consider the regularized heat trace in a single degree.

\begin{align*}
\Tr^R e^{-t\Delta_k} = \dashint_0^\infty \int_B \tr (e^{-t\Delta_k})\,   r^bdr \dVol_{g_B}
=: \dashint_0^\infty r^b H'_k(t,r) dr,
\end{align*}
where we recall the notation for a regularized integral
$$
\dashint_0^\infty := \LIM_{\varepsilon \to 0} \LIM_{R\to \infty} \int_\varepsilon^R.
$$
Using \eqref{rescaling}, we obtain 
\begin{align*}
\Tr^R e^{-t\Delta_k} = \dashint_0^\infty H_k'(t,r) dr = 
t^{-\frac{b+1}{2}} \dashint_0^\infty H_k'\Bigl(1,\frac{r}{\sqrt{t}}\Bigr) dr.
\end{align*}
There is a variable change rule for regularized integrals, see e.g.
\cite[Lemma 2.1.4]{LeschBook}. Namely, if $f:(0,\infty) \to (0,\infty)$ admits polyhomogeneous 
expansions as $x\to 0$ and $x\to \infty$, so that its regularized integral exists, then
for any $\lambda>0$
\begin{equation}
\dashint_0^\infty f(\lambda x)\, dx=\frac{1}{\lambda} \dashint_0^\infty f(x)\, dx 
+\frac{1}{\lambda} \sum_{\ell=1}^{n_f} c_\ell \log(\lambda)^\ell
\label{eq:ChangeofVar}
\end{equation}
for some coefficients $c_k$ and $n_f \in \N$, which can be explicitly given 
in terms of the coefficients in the asymptotic expansions of $f$. Consequently
\begin{align}
\Tr^R e^{-t\Delta_k} = t^{-\frac{b+1}{2}} \dashint_0^\infty H_k'\Bigl(1,\frac{r}{\sqrt{t}}\Bigr) r^b dr
= \dashint_0^\infty H_k'\Bigl(1,r \Bigr)r^b dr + \sum_{\ell=1}^{n_k} c_{k,\ell} \log(t)^\ell.
\label{eq:regTraceCone}
\end{align}
By definition \eqref{zeta-formula} we find
\begin{align*}
\Gamma(s)\zeta(s, \Delta_k) &=  \dashint_0^\infty H_k'\Bigl(1,r \Bigr)r^b dr 
\cdot \Bigl( \int_0^1 t^{s-1} dt + \int_1^{\infty} t^{s-1} dt \Bigr) \\ &+
\sum_{\ell=1}^{n_k} c_{k,\ell} \cdot \Bigl( \int_0^1 t^{s-1} \log(t)^\ell dt + \int_1^{\infty} t^{s-1} \log(t)^\ell dt \Bigr),
\end{align*}
where we of course mean analytic continuations of the four integrals.
For any $(\alpha,k)\in \C\times \N_0$, and $f(x)=x^\alpha \log(x)^k$ \cite[Equation 1.12]{LeschBook} asserts that (and one can easily check this)
\[\int_0^1 f(x)x^{s-1}\, dx=(-1)^k \frac{k!}{(\alpha+s)^{k+1}},\]
\[\int_1^{\infty} f(x)x^{s-1}\, dx=(-1)^{k+1}\frac{k!}{(\alpha+s)^{k+1}}.\]
So we get the somewhat strange looking formula.
\begin{equation}
\int_0^1 f(x)x^{s-1}\, dx+\int_1^{\infty} f(x)x^{s-1}\, dx=0.
\label{eq:RegIntZero}
\end{equation}
This proves $\zeta(s, \Delta_k) \equiv 0$ for each $k$ and the claim 
follows for the special case $F=\{pt\}$. \medskip

\noindent\textbf{Case 2 - $F$ arbitrary:} \medskip

\noindent
We start by deriving a partial product rule for the analytic torsion following \cite[Theorem 2.5]{RaySin:RTA} and \cite[Proposition 2.3]{Les:GFA}. Since the metric on $\mathscr{U}$ is of product type, $g_{\cU}=g_{\mathscr{C}(B)}+g_F$, the Laplacian (acting on forms) splits,
\[\Delta_{\cU}=\Delta_{\mathscr{C}(B)}+\Delta_F.\]
Splitting $\Omega^*(\cU)=\Omega^*(\mathscr{C}(B))\otimes \Omega^*(F)$, we get a corresponding split of the heat operators (and their integral kernels) 
\[e^{-t\Delta_{\cU}}=e^{-t\Delta_{\mathscr{C}(B)}}\otimes e^{-t\Delta_F}.\]
Finally, we also split the number operator $N$ as $N=N_{\mathscr{C}(B)}+N_F$, where $N_{\mathscr{C}(B)}\coloneqq N_{\mathscr{C}(B)}\otimes \mathds{1}$ and $N_F\coloneqq \mathds{1}\otimes N_F$. One readily checks that if $V,W$ are vector spaces and $A\in \mathrm{End}(V), B\in \mathrm{End}(W)$, then $\tr_{V\otimes W}(A\otimes B)=\tr_V(A)\cdot \tr_W(B)$. So the pointwise trace satisfies
\begin{align}
&\tr\left((-1)^N N e^{-t\Delta_{\cU}}\right)=\tr\left((-1)^{N_{\mathscr{C}(B)}}(-1)^{N_F} (N_{\mathscr{C}(B)}+N_F) e^{-t\Delta_{{\mathscr{C}(B)}}}\otimes e^{-t\Delta_F}\right) \notag\\
&=\tr\left((-1)^{N_{\mathscr{C}(B)}} N_{{\mathscr{C}(B)}} e^{-t\Delta_{{\mathscr{C}(B)}}}\right)\cdot\tr\left((-1)^{N_F} e^{-t\Delta_F}\right) \notag \\
&+\tr\left((-1)^{N_{\mathscr{C}(B)}}  e^{-t\Delta_{{\mathscr{C}(B)}}}\right)\cdot\tr\left((-1)^{N_F} N_F e^{-t\Delta_F}\right). 
\label{eq:traceSum}
\end{align}
Here the traces on the first line are over $\Omega^*(\cU)$, whereas the traces on line 2 (and 3) are over $\Omega^*( \mathscr{C}(B))$ and $\Omega^*(F)$ respectively.
Since $F$ is compact, we can integrate over $F$ on both sides of \eqref{eq:traceSum}. 
By the McKean-Singer formula \cite[Section 6, eq. 1]{McKeanSing} we have
\[\int_F \tr\left( (-1)^{N_F} e^{-t\Delta_F}\right) \, \dVol_{g_F}=\chi(F,E)\]
for any $t>0$, where $E$ really means $E_{\vert F}$. In particular, it does not depend on $t$. So after taking the regularised heat trace of $\mathscr{C}(B)$ and the regularized time integral, the second line of \eqref{eq:traceSum} will become
\[
\Gamma(s) \sum_{k=0}^{b+1} (-1)^k \cdot k \cdot \zeta (s,\Delta_{\mathscr{C}(B),k}) \cdot \chi(F,E).
\]
This vanishes by case 1 since each $\zeta (s,\Delta_{\mathscr{C}(B),k})=0$. \medskip

We therefore turn to the second term in \eqref{eq:traceSum}. We perform the regularized spatial integral over $\mathscr{C}(B)$ and the integral over $F$. Using \eqref{eq:regTraceCone} we find
\begin{align*}
& \Tr^R\left((-1)^{N_{\mathscr{C}(B)}}  e^{-t\Delta_{{\mathscr{C}(B)}}}\right)\cdot\Tr\left((-1)^{N_F} N_F e^{-t\Delta_F}\right)\\
&=\sum_{k=0}^{b+1} (-1)^k  \left( \dashint_0^{\infty} H'_k(1,r)\, r^b dr+\sum_{\ell=1}^{n_k}c_{k\ell} \log(t)^{\ell}\right)\Tr\left((-1)^{N_F} N_F e^{-t\Delta_F}\right).
\end{align*}
We abbreviate a bit. Set $c_{k\ell}=0$ for $k>n_k$, let $n\coloneqq \max_k n_k$, $C_0\coloneqq \sum\limits_{k=0}^{b+1} (-1)^k   \dashint_0^{\infty} H'_k(1,r)\, r^b dr$, and $C_{\ell}\coloneqq \sum\limits_{k=0}^{b+1} (-1)^k c_{k\ell}$. Then 
\begin{align*}
& \Tr^R\left((-1)^{N_{\mathscr{C}(B)}}  e^{-t\Delta_{\mathscr{C}(B)}}\right)\cdot\Tr\left((-1)^{N_F} N_F e^{-t\Delta_F}\right)\\
&=\sum_{\ell=0}^{n} C_{\ell} \log(t)^{\ell} \, \Tr\left((-1)^{N_F} N_F e^{-t\Delta_F}\right),
\end{align*}

The trace over the heat kernel $e^{-t\Delta_F}$ is really over all eigenvalues, but we can change it to a sum over only the positive eigenvalues. The reason is that the difference is given by
\[\sum_{\ell=0}^{n} C_{\ell} \log(t)^{\ell}\sum_{j=0}^{\dim(F)} (-1)^j \, j \, b_j(F,E),\]
where $b_j(F,E)\coloneqq \dim(\mathrm{Ker}(\Delta_{F,j}))$.
Multiplying this by $t^{s-1}$ and performing the integrals in time yields zero by \eqref{eq:RegIntZero}.  As such, we are left with studying
\[\sum_{\ell=0}^{n} C_{\ell}\log(t)^{\ell} \, \Tr'\left((-1)^{N_F} N_F e^{-t\Delta_F}\right),\]
where $Tr'$ means the trace but omitting zero eigenvalues.
We have to compute
\[\int_0^1 t^{s-1}  \log(t)^{\ell} \sum_{j=0}^{\dim(F)} \sum_{\substack{\lambda_j \in \spec(\Delta_{F,j})\\ \lambda_j>0}} (-1)^j j e^{-t\lambda_{j}}\, dt\] 
and
\[\int_1^\infty t^{s-1}  \log(t)^{\ell} \sum_{j=0}^{\dim(F)} \sum_{\substack{\lambda_j \in \spec(\Delta_{F,j})\\ \lambda_j>0}} (-1)^j j e^{-t\lambda_{j}}\, dt,\]
where the sums over the eigenvalues are with multiplicity.
Both integrals converge for $\text{Re}(s)>1$, so we can combine them into one integral. One readily checks
\[\int_0^\infty t^{s-1}e^{-\lambda t} \log(t)^{\ell}\, dt=\frac{d^{\ell}}{ds^{\ell}}(\lambda^{-s}\Gamma(s)).\]
Hence our integral reads
\begin{align*}
&\int_0^\infty t^{s-1} \sum_{\ell=0}^n C_{\ell} \log(t)^{\ell} \sum_{j=0}^{\dim(F)} \sum_{\substack{\lambda_j \in \spec(\Delta_{F,j})\\ \lambda_j>0}} (-1)^j j e^{-t\lambda_{j}}\, dt\\
&=\sum_{\ell=0}^n C_{\ell} \frac{d^{\ell}}{ds^{\ell}} \left(\Gamma(s)\sum_{j=0}^{\dim(F)}\sum_{\substack{\lambda_j \in \spec(\Delta_{F,j})\\ \lambda_j>0}} (-1)^j j \lambda_{j}^{-s}\right) \\
&=\sum_{\ell=0}^n C_{\ell} \frac{d^\ell}{ds^{\ell}}\left(\Gamma(s)\sum_{j=0}^{\dim(F)} (-1)^j j \zeta(s,\Delta_{F,j})\right).
\end{align*} 
But the sum $\sum\limits_{j=0}^{\dim(F)} (-1)^j j \zeta(s,\Delta_{F,j})$ vanishes by \cite[Theorem 2.3]{RaySin:RTA} since $\dim(F)$ is even.
Hence the torsion of $\mathscr{C}(B)\times F$ is trivial.

\end{proof}
\begin{Rem}
One could possibly drop the assumption that the dimension of $F$ is even. The obstruction in the above argument is the term 
\[\Tr^R\left((-1)^{N_{\mathscr{C}(B)}} e^{- t\Delta_{\mathscr{C}(B)}}\right)\eqqcolon \chi^R(\mathscr{C}(B),t).\]
If one can show that is time-independent like in the compact case, the product formula would read (suppressing the metric from the notation)
\[\log(T(\mathscr{C}(B)\times F,E))=\log(T(\mathscr{C}(B),E))\chi(F,E)+\chi^R \log(T(F,E))=\chi^R(\mathscr{C}(B)) \log(T(F,E)).\]
We do not have an argument that $\chi$ is time-independent in general. When $B=\S^b/\Gamma$ for some finite group $\Gamma\subset O(b+1)$ acting freely and $g_B$ is the round metric, then $\chi^R(\mathscr{C}(B),t)=0$ for all $t$. To see this, assume $B=\S^b$. Then $\mathscr{C}(B)\cong \R^{b+1}\setminus\{0\}$ with the Euclidean metric. The heat kernel acting on $k$-forms is $H_k=H_0 \mathds{1}_{\Omega^k}$, and $H_0(t,x,y)=(4\pi t)^{-\frac{b+1}{2}}\exp\left(-\frac{\vert x-y\vert }{4t}\right)$. So $\tr(H_k)=c_k t^{-\frac{b+1}{2}}$ for all $r$ and the regularized heat trace vanishes by \eqref{eq:RegIntZero}. When $B=\S^b/\Gamma$, the heat kernel is the same since it is rotationally invariant and therefore descends to the quotient. 
\end{Rem}

\providecommand{\href}[2]{#2}
 \bigskip

\end{document}